\documentclass[12pt]{article}

\usepackage{amsmath,amsfonts,amsthm,amssymb}
\usepackage[mathscr]{eucal}

\theoremstyle{definition}
\newtheorem{definition}{Definition}
\theoremstyle{remark}
\newtheorem{remark}[definition]{Remark}
\newtheorem{example}[definition]{Example}

\newtheoremstyle{mytheorem}{0.5cm}{0.2cm}{\slshape}{ }{\bfseries}{.}{ }{}
\theoremstyle{mytheorem}
\newtheorem{theorem}[definition]{Theorem}
\newtheorem{prop}[definition]{Proposition}

\newtheorem{cor}[definition]{Corollary}

\newcommand{\stsets}[1]{\mathbb{#1}}
\newcommand{\R}{\stsets{R}}
\newcommand{\Rd}{\R^d}
\renewcommand{\SS}{{\stsets{S}}}
\newcommand{\SSp}{{\stsets{S}^{d-1}_+}}
\newcommand{\Sphere}[1][d-1]{\stsets{S}^{#1}}

\newcommand{\sI}{\mathcal{I}}

\newcommand{\sL}{\mathcal{L}}
\newcommand{\sS}{\mathcal{S}}

\newcommand{\Prob}[1]{\mathbf{P}\{#1\}}
\DeclareMathOperator{\E}{{\bf E}}

\DeclareMathOperator{\grad}{grad}

\DeclareMathOperator{\sign}{sign}
\DeclareMathOperator{\Vol}{Vol}
\DeclareMathOperator{\one}{{ 1\hspace*{-0.55ex}I}}
\newcommand{\Ind}[1]{\one_{#1}}

\newcommand{\sas}{S\alpha S}
\newcommand{\saps}[1][\alpha']{S #1 S}
\newcommand{\deq}{\overset{\scriptscriptstyle\mathcal{D}}{=}}
\newcommand{\spsum}[1][p]{\stackrel{\sim}{+}_#1}
\newcommand{\psum}[1][\alpha]{+_{#1}}
\newcommand{\spower}[1]{{\langle #1\rangle}}
\newcommand{\uF}{\|u\|_F}
\newcommand{\I}{I\!}
\newcommand{\thf}{\frac{1}{2}}
\newcommand{\ddashv}{\dashv_S}
\newcommand{\pnn}[1]{{{|\!|\!|}#1{|\!|\!|}}_\alpha}
\newcommand{\Gf}[1]{\Gamma(\frac{#1}{2})}
\newcommand{\thaf}[1]{\frac{#1}{2}}

\newcommand{\eps}{\varepsilon}
\renewcommand{\phi}{\varphi}
\newcommand{\ti}{\to\infty}

\newcommand{\dto}{\overset{{\mathrm{d}}\ }{\rightarrow}}

\newlength{\querylen}
\usepackage{fancybox}

\numberwithin{equation}{section}
\numberwithin{definition}{section}

\begin{document}
\bibliographystyle{plain}

\title{Convex and star-shaped sets associated with multivariate stable
  distributions}
\author{\textsc{Ilya Molchanov}\\
  \normalsize
  Department of Mathematical Statistics and Actuarial Science,\\
  \normalsize
  University of Bern, Sidlerstrasse 5, CH-3012 Bern, Switzerland\\
  \normalsize E-mail: ilya@stat.unibe.ch }
\maketitle

\begin{abstract}
  \noindent 
  It is known that each symmetric stable distribution in $\R^d$ is
  related to a norm on $\R^d$ that makes $\R^d$ embeddable in
  $L_p([0,1])$. In case of a multivariate Cauchy distribution the unit
  ball in this norm is the polar set to a convex set in $\R^d$ called
  a zonoid. This work interprets general stable laws using convex or
  star-shaped sets and exploits recent advances in convex geometry in
  order to come up with new probabilistic results for multivariate
  stable distributions. In particular, it provides expressions for
  moments of the Euclidean norm of a stable vector, mixed moments and
  various integrals of the density function. It is shown how to use
  geometric inequalities in order to bound important parameters of
  stable laws.  Furthermore, covariation, regression and orthogonality
  concepts for stable laws acquire geometric interpretations.  A
  similar collection of results is presented for one-sided stable
  laws.

  \medskip


\end{abstract}

\tableofcontents

\newpage

\section{Introduction}
\label{sec:introduction}

Since P.~L\'evy it is well known that the characteristic function of a
symmetric stable law in $\R^d$ can be represented as the exponential
of the norm as $\phi(u)=e^{-\|u\|^p}$, where $p$ is the characteristic
exponent of the stable law. It is also well known
\cite{bret:dac:kriv66,her63} that all norms which might appear in this
representation make $(\R^d,\|\cdot\|)$ isometrically embeddable in the
space $L_p([0,1])$ if $p\geq1$. The corresponding result holds also
for all $p\in(0,2]$, see \cite[Lemma~6.4]{kold05}.

Each \emph{norm} in $\R^d$ gives rise to the corresponding \emph{unit
  ball} $F$.  In this paper we neglect the convexity property of the
norm and use the term norm, where other works sometimes use the term
\emph{gauge function}. If the norm is convex, it can be realised as
the support function $\|u\|=h(K,u)$ for the set $K$ polar to the unit
ball $F$, see Section~\ref{sec:star-bodies-convex}. It is known from
convex geometry (see \cite{schn} for the standard reference) that unit
balls in spaces embeddable in $L_1([0,1])$ are exactly polar sets to
\emph{zonoids}. Recall that zonoids appear as Hausdorff limits of
zonotopes, i.e.  finite Minkowski sums of segments. In application to
stable laws with characteristic exponent $p=1$, we have that
$\phi(u)=e^{-h(K,u)}$ is a characteristic function (necessarily
corresponding to the multivariate Cauchy distribution) if and only if
$K$ is a zonoid. It is known that all planar centrally symmetric
convex sets are zonoids, while this is no longer the case in
dimensions 3 and more. This corresponds to a result of Ferguson
\cite{fer62}, who showed that the dependency structure of symmetric
\emph{bivariate} Cauchy distributions can be described using
\emph{any} norm in $\R^2$, while such representation is no longer
possible for \emph{all} norms in dimensions three and more.

It is explained in Section~\ref{sec:star-bodi-assoc} that the
correspondence between norms and stable laws can be extended to
include all symmetric stable laws by representing their characteristic
functions using (possibly non-convex) norms. The unit ball $F$ in the
corresponding norm is a star-shaped set called the star body
associated with the stable law. Thus, each symmetric stable
distribution is uniquely determined by a star body $F$ and the value
of the characteristic exponent $\alpha$ and the paper aims to show
that this geometric interpretation is useful in the studies of
multivariate stable laws. In particular, $F$ is an ellipsoid if
and only if the underlying distribution is sub-Gaussian.
Section~\ref{sec:zono-assoc-sas} shows that if the characteristic
exponent $p$ is at least one, then the norms are convex and the
support function representation is also possible using the associated
zonoid $K$ being the polar set to $F$. This associated zonoid $K$ is
called $L_p$-zonoid, since $K$ can be represented as the Hausdorff
limit of a power sum of segments, where the support function of $K$ is
the $p$-mean of the support functions of the summands.

While the general correspondence between zonoids and stable laws is
well understood, this paper concentrates on further relationships
between probabilistic aspects of stable laws and geometric properties
of associated star-shaped and convex sets.  The core of the paper
begins in Section~\ref{sec:symm-stable-dens}, where it is shown how to
relate the value of the probability density function $f$ of the
symmetric stable law at zero to the volume of the associated star body
$F$. It is shown how derivatives of $f$ at the origin are related
to further geometric properties of $F$, in particular to certain
ellipsoids associated with $F$. It also provides an expression for the
R\'enyi entropy of symmetric stable laws.  Using geometric results on
approximation of convex sets with ellipsoids,
Section~\ref{sec:symm-stable-dens} ends up with a result that gives an
estimate for the quality of approximation of a symmetric stable law
with a sub-Gaussian one.

Section~\ref{sec:moments-sas-laws} uses the Fourier analysis for
generalised functions together with the geometric representation of
the characteristic function in order to compute a number of important
characteristics of symmetric stable laws. These characteristics
include the moments of the norm of a stable random vector, which
previously were known only in the isotropic case, mixed moments of
(possibly signed) powers of the coordinates, integrals of the density
over subspaces, etc. This section also presents a number of
inequalities for the moments and settles the equality cases. Finally,
it clarifies a relationship between zonoids of stable laws and zonoids
of random vectors studied in \cite{mos02}.

Section~\ref{sec:power-sums-maxima} deals with one-sided strictly
stable laws supported by $\R_+^d$. It first establishes a geometric
characterisation of stable laws for power sums, which fill the gap
between the arithmetic addition and the coordinatewise maximum scheme
for random vectors. Note that relationships between \emph{max-stable}
random vectors and convex sets have been explored in \cite{mol06}. It
is known that max-stable distributions with unit Fr\'echet marginals
are exactly those having the cumulative distribution function
$F(u^{-1})=e^{-h(K,u)}$, where $u^{-1}$ is the vector composed of the
reciprocals of $u$ and $K$ is a \emph{max-zonoid}, i.e.  the
expectation of a random crosspolytope. It is shown in Section
\ref{sec:power-sums-maxima} that stable laws for power sums (and also
one-sided strictly stable laws) correspond to a new family of convex
sets called $L_1(p)$-zonoids.  These sets appear as (set-valued)
expectations of randomly rescaled $\ell_q$-balls in $\R^d$, where $q$
is reciprocal to $p$. Finally, it provides expressions for some
moments of one-sided strictly stable laws.

The star bodies and zonoids associated with stable laws are determined
by the spectral measures of stable laws.
Section~\ref{sec:surf-area-repr} shows that under quite general
conditions, the spectral measures themselves admit a geometric
interpretation as, e.g., surface area measures of further auxiliary
convex sets called spectral bodies of stable laws. It shows how
geometric inequalities can be used to relate volumes of the
corresponding convex sets, and thereupon derive bounds for densities
and moments of stable laws.

The geometric interpretation of the covariation is given in
Section~\ref{sec:covariation}.  It also discusses the regression
problem for symmetric stable laws, in particularly, the linearity
property of multiple regression, which goes back to W.~Blaschke's
characterisation theorem for ellipsoids.

Section~\ref{sec:oper-with-assoc} describes several operations with
associated star bodies and zonoids and their probabilistic meaning.
Using recent approximation results from convex geometry, it is proved
that each symmetric stable law can be obtained as the limit for sums
of sub-Gaussian laws.  It also discusses optimisation ideas, which
appear, e.g. in optimising a portfolio whose components have jointly
stable distribution.

Section~\ref{sec:james-orthogonality} discusses the concept of James
orthogonality for symmetric stable random variables, which is also
extended to define orthogonality of symmetric stable random vectors.

This work attempts to highlight novel relationships between convex
geometry and the theory of stable distributions. Further developments
are surely possible by invoking other recent results on isotropic
bodies, geometric concentration inequalities, or properties of convex
sets in spaces of high (but finite) dimension. It is possible to apply
the results to deal with finite-dimensional distributions of stable
processes or work directly in a general infinite dimensional setting.

\section{Star bodies and convex sets}
\label{sec:star-bodies-convex}

A set $F$ in $\R^d$ is \emph{star-shaped} if $[0,u]\subset F$ for each
$u\in F$.  A closed bounded set $F$ is called a \emph{star body} if
for every $u\in F$ the interval $[0,u)$ is contained in the interior
of $F$ and the \emph{Minkowski functional} (or the gauge function) of
$F$ defined by
\begin{displaymath}
  \|u\|_F=\inf\{s\geq0:\; u\in sF\} 
\end{displaymath}
is a continuous function of $u\in\R^d$.   The set $F$ can be
recovered from its Minkowski functional by
\begin{displaymath}
  F=\{u:\; \|u\|_F\leq 1\}\,,
\end{displaymath}
while the \emph{radial function}
\begin{displaymath}
  \rho_F(u)=\|u\|_F^{-1}
\end{displaymath}
provides the polar coordinate representation of the boundary of $F$
for $u$ from the unit \emph{Euclidean sphere} $\Sphere$.  In the
following we usually consider origin-symmetric star-shaped sets and
call them \emph{centred} in this case. If the star body $F$ is centred
and \emph{convex}, then $\|u\|_F$ becomes a convex norm on $\R^d$.

The \emph{$\ell_p$-ball} in $\R^d$ is defined by
\begin{displaymath}
  B_p^d=\{x\in\R^d:\; \|x\|_p\leq1\}\,,
\end{displaymath}
where $\|x\|_p=(|x_1|^p+\cdots+|x_d|^p)^{1/p}$ for $p\neq0$, i.e.
$\|x\|_p=\|x\|_F$ with $F=B_p^d$.

If $p\in \R\setminus\{0\}$, the \emph{$p$-star sum} of two star bodies
$F_1$ and $F_2$ is defined by its Minkowski functional as
\begin{equation}
  \label{eq:x_k_1sps-k_2=x_k_1p+}
  \|u\|_{F_1\spsum F_2}=(\|u\|_{F_1}^p+\|u\|_{F_2}^p)^{1/p}\,,\quad
  u\in\R^d\,.
\end{equation}
This definition goes back to Firey \cite{fir61} and was later
investigated by Lutwak \cite{lut96}.  For $p=-1$ we obtain the radial
sum, i.e. the radial function of the result is the sum of two radial
functions of the summands.  Extended by the limit for $p=-\infty$, the
$p$-sum yields the union $F_1\cup F_2$ and the intersection $F_1\cap
F_2$ for $p=\infty$.  Note that (\ref{eq:x_k_1sps-k_2=x_k_1p+}) means
that the Minkowski functional of $F_1\spsum F_2$ is proportional to
the $p$-mean of the Minkowski functionals of $F_1$ and $F_2$. See
\cite{har:lit:pol34} for a comprehensive study of $p$-means of real
numbers.

A convex set $K$ in $\R^d$ is called a \emph{convex body} if $K$ is
compact and has non-empty interior. We usually use the letter $F$ for
star bodies (which are not necessarily convex) and $K$ for convex
bodies.

The \emph{support function} of a bounded set $A$ in $\R^d$ is defined
by
\begin{equation}
  \label{eq:supf}
  h(A,u)=\sup\{\langle x,u\rangle:\; x\in A\}\,,\quad u\in\R^d\,.
\end{equation}
Clearly, $h(A,u)$ coincides with the support function of the convex
hull of $A$. Any function $f:\R^d\to\R$, which is sublinear, i.e.
$f(cx)=cf(x)$ for all $c\geq 0$ and $f(x+y)\leq f(x)+f(y)$ for all
$x,y\in\R^d$, is a support function of a convex compact set, see
\cite[Th.~1.7.1]{schn}. 

The \emph{polar set} to a convex body $K$ is defined by
\begin{equation}
  \label{eq:k=u:-hk-uleq1}
  K^*=\{u:\; h(K,u)\leq1\}\,.
\end{equation}
The same definition applies if $K$ is not necessarily convex.  
If $K$ is convex, then $F=K^*$ is also convex and
\begin{displaymath}
  \|u\|_F=h(K,u)\,,\quad u\in\R^d\,,
\end{displaymath}
i.e. the Minkowski functional of $F$ is the support function of $K$.

The \emph{Firey $p$-sum} of convex sets $K_1$ and $K_2$ that both
contain the origin can be defined for $p\geq1$ as the convex set
$L=K_1\psum[p] K_2$ with the support function
\begin{equation}
  \label{eq:psum}
  h(L,u)=(h(K_1,u)^p+h(K_2,u)^p)^{1/p}\,,
\end{equation}
see \cite{fir62} and \cite{lut93}. 
The Firey sum is closely related to the $p$-star sum
(\ref{eq:x_k_1sps-k_2=x_k_1p+}), since
\begin{displaymath}
  (K_1\psum[p] K_2)^* =K^*_1\spsum K^*_2
\end{displaymath}
for convex $K_1$ and $K_2$ that contain the origin.  If $p=1$, the
Firey sum turns into the \emph{Minkowski sum} defined as
\begin{displaymath}
  K_1+K_2=\{x_1+x_2:\; x_1\in K_1,\, x_2\in K_2\}\,.
\end{displaymath}
Then 
\begin{displaymath}
  h(K_1+K_2,u)=h(K_1,u)+h(K_2,u)\,,\quad u\in\R^d\,.
\end{displaymath}

Further $\|x\|$ (without subscript) denotes the \emph{Euclidean norm}
of $x\in\R^d$.  The (Euclidean) norm of a set $K$ is defined as
$\|K\|=\sup\{\|u\|:\; u\in K\}$.  By $\Vol_d(K)$ or $|K|$ we denote
the $d$-dimensional \emph{Lebesgue measure} of $K$. The volume of the
unit $\ell_2$-ball (i.e. Euclidean ball) $B$ in $\R^d$ is denoted by
\begin{displaymath}
  \kappa_d=|B|=\frac{\pi^{d/2}}{\Gamma(1+\frac{d}{2})}\,,
\end{displaymath}
where $\Gamma$ is the \emph{Gamma function}. The same expression for
$\kappa_d$ is used also for all real $d>0$.

\medskip

A \emph{random closed set} in $\R^d$ is a random element in the space
of closed sets equipped with the Fell topology and the corresponding
Borel $\sigma$-algebra, see \cite{mo1}.  A random closed set $X$ is
said to be \emph{compact} if $X$ has a.s. compact realisations. If $X$
is a random compact set in $\R^d$ such that $\|X\|$ is integrable,
then $\E h(X,u)$ is the support function of a convex compact set
called the (selection or Aumann) \emph{expectation} of $X$ and denoted
by $\E X$, see \cite[Sec.~2.1]{mo1}.  If $X$ is a simple random convex
set, i.e.  $X$ takes only a finite number of convex compact values
$K_1,\dots,K_n$ with probabilities $p_1,\dots,p_n$, then $\E
X=p_1K_1+\cdots+p_nK_n$.  The expectation of a general $X$ can be
obtained by approximating $X$ with simple random sets, see
\cite[Th.~2.1.21]{mo1}.

By applying (\ref{eq:psum}) to simple random sets it is possible to
define the \emph{Firey $p$-expectation} $\E_p X$, $p\geq1$, of a
random compact set $X$ such that $0\in X$ almost surely and
$\E\|X\|^p<\infty$. In particular,
\begin{displaymath}
  h(\E_p X,u)=\left(\E [h(X,u)^p]\right)^{1/p}
\end{displaymath}
is the $p$-mean of $h(X,u)$ for $p\geq1$, see also \cite{fir67}.

Sets that appear as finite Minkowski sums of segments are called
zonotopes. \emph{Zonoids} are limits of zonotopes in the Hausdorff
metric, i.e. they can be represented as expectations of random
segments. By changing the Minkowski sum to the Firey $p$-sum with
$p\geq1$ one obtains \emph{$L_p$-zonoids}, which appear as limits for
Firey $p$-sums of centred segments. This generalisation in the
geometric context has been first mentioned in \cite{goo:w93} and has
been thoroughly investigated in \cite{lut:yan:zhan04v}. It should be
noted that $L_p$-zonoids are exactly those sets that appear as polar
sets to the unit balls in spaces isometric to a $d$-dimensional
subspace of $L_p([0,1])$, see \cite{kol96p}.

If $X$ is a random closed set with almost surely star-shaped
realisations (i.e.  random star-shaped set), then the \emph{$p$-star
  expectation} of $X$ is the star-shaped set $F=\E^*_p X$ whose
Minkowski functional is given by
\begin{displaymath}
  \|u\|_F=(\E \|u\|_X^p)^{1/p}\,.
\end{displaymath}
This expectation defines a star-shaped set for all $p\neq0$.  Note
however that $F$ is not necessarily a star body, since $\|u\|_F$ may
be infinite. If $X$ is a random convex body that contains the origin
and satisfies $\E\|X\|^p<\infty$, then
\begin{equation}
  \label{eq:e_p-x=e_p-x}
  \E_p X=\E_p^* (X^*)\,,\quad p\geq1\,.
\end{equation}

\section{Star bodies associated with $\sas$ distributions}
\label{sec:star-bodi-assoc}

A random vector $\xi\in\R^d$ is called \emph{symmetric
  $\alpha$-stable} (notation $\sas$) if $\xi$ coincides in
distribution with $-\xi$ and, for all $a,b>0$,
\begin{displaymath}
  a^{1/\alpha}\xi_1+b^{1/\alpha}\xi_2\deq
  (a+b)^{1/\alpha}\xi\,, 
\end{displaymath}
where $\xi_1,\xi_2$ are independent copies of $\xi$, and $\deq$
denotes equality in distribution. The value of $\alpha$ is called the
\emph{characteristic exponent} of $\xi$. It is well known that $\xi$
is normally distributed if and only if  $\alpha=2$.

\begin{theorem}[see Th.~2.4.3 \cite{sam:taq94}]
  \label{th:sam-taq}
  A random vector $\xi$ is $\sas$ with $\alpha\in(0,2)$ if and only if
  there exists a unique symmetric finite measure $\sigma$ on the unit
  sphere $\SS$ in $\R^d$ such that the characteristic function of
  $\xi$ is given by
  \begin{equation}
    \label{eq:sam-taq}
    \phi_\xi(u)=\E e^{i \langle \xi,u\rangle}
    =\exp\left\{-\int_{\SS} |\langle u,z\rangle|^\alpha
      \sigma(dz)\right\}\,.
  \end{equation}
\end{theorem}

The measure $\sigma$ is called the \emph{spectral measure} of $\xi$.
Representation (\ref{eq:sam-taq}) holds also for $\alpha=2$, although
the spectral measure is not necessarily unique in this case.  Although
the sphere $\SS$ can be defined with respect to any chosen (reference)
norm in $\R^d$, in this paper we only use the Euclidean reference
norm, so that $\SS=\Sphere$ is the Euclidean sphere in $\R^d$.
\renewcommand{\SS}{\Sphere}

The expression in the exponential in the right hand side of
(\ref{eq:sam-taq}) is an even homogeneous (of order $\alpha$) function
of $u$ and so defines the Minkowski functional of a centred
star body $F$ as
\begin{equation}
  \label{eq:u_k=int_ss-langle-u}
  \|u\|_F^\alpha=\int_{\SS} |\langle u,z\rangle|^\alpha \sigma(dz)\,,
\end{equation}
so that
\begin{equation}
  \label{eq:ch-f-rep}
  \phi_\xi(u)=e^{-\|u\|_F^\alpha}\,, \quad u\in\R^d\,.
\end{equation}
The star body $F$ is called the \emph{associated star body of $\xi$}.
Since $\|u\|_F$ is finite, $F$ always contains a neighbourhood of the
origin. If $\sigma$ is not concentrated on a great sub-sphere of $\SS$
(i.e.  the intersection of $\SS$ with a $(d-1)$-dimensional subspace),
then $\xi$ is called \emph{full-dimensional}. In this case $\|u\|_F>0$
for all $u\neq0$.  

The right-hand side of (\ref{eq:u_k=int_ss-langle-u}) is called the
\emph{$\alpha$-cosine transform} of $\sigma$, which is studied also
for all $\alpha>-1$, see \cite{grin:zhan99}, where it is shown that
$\sigma$ is unique if $\alpha$ is not an even integer.  Although
$\sigma$ is not unique for $\alpha=2$, the star body associated with
the normal law is a unique ellipsoid.

In Section~\ref{sec:zono-assoc-sas} we see that $F$ is convex if
$\alpha\in[1,2]$. Example~\ref{ex:substable} describes $\sas$ laws
with $\alpha<1$ whose associated star bodies are convex. The following
simple examples deal with general $\alpha\in(0,2]$.

\begin{example}[Complete independence]
  \label{ex:indep}
  If $\xi$ is $\sas$ with independent components, then
  \begin{displaymath}
    \phi_\xi(u) =e^{-\|u\|_\alpha^\alpha}\,,
  \end{displaymath}
  i.e. the star body $F$ associated with $\xi$ is $\ell_\alpha$-ball
  $B_\alpha^d$.  This star body is not convex if $\alpha<1$.
\end{example}

\begin{example}[Complete dependence]
  \label{ex:complete-dep}
  If $\xi=(\xi_1,\dots,\xi_1)$ for $\sas$ random variable $\xi_1$,
  then
  \begin{displaymath}
    \phi_\xi(u)
    =\exp\left\{-\left|\sum u_i\right|^\alpha\right\}\,.
  \end{displaymath}
  Note that $\xi$ is not full-dimensional, so that
  $\rho_F(u)=\|u\|_F^{-1}=\left|\sum u_i\right|^{-1}$ is the radial
  function of an unbounded star body $F$. The corresponding polar set
  $F^*$ is the segment with end-points $\pm(1,\dots,1)$.
\end{example}

The associated star body of $\sas$ random vector $\xi$ can be obtained
as $F=c^{-1/\alpha}\E_\alpha^* Y_{\eta}$, i.e. $F$ is the star
expectation of
\begin{displaymath}
  Y_{\eta}=\{x:\; |\langle x,\eta\rangle|\leq1\}\,,
\end{displaymath}
where $c$ is the total mass of $\sigma$ and $\eta$ is distributed
according to $c^{-1}\sigma$. For this, it suffices to note that
$|\langle u,z\rangle|=\|u\|_{Y_z}$ where $Y_z$ is the polar set to
$[-z,z]$.  It is often useful to redefine the spectral measure
$\sigma$ to be a probability measure on the whole $\R^d$. Then no
constant $c$ is needed, so that $F=\E_\alpha^* Y_\eta$, where $\eta$
is distributed in $\R^d$ according to $\sigma$.  Note that it is
always possible to extend $\sigma$ to be a square integrable on
$\R^d$, so that the integral (\ref{eq:u_k=int_ss-langle-u}) exists for
all $\alpha\in(0,2]$.

A centred convex body $F$ in $\R^d$ is called an \emph{$L_p$-ball} if
it is the unit ball of a $d$-dimensional subspace of $L_p([0,1])$.
Denote by $\sL_p$ the family of $L_p$-balls.  It is known that
$F\in\sL_p$ if and only if
\begin{equation}
  \label{eq:u_fp=-langle-u}
  \|u\|_F^p=\int_{\Sphere} |\langle u,z\rangle|^p\mu(dz)
\end{equation}
for a finite measure $\mu$ on $\Sphere$, see
\cite[Lemma~4.8]{grin:zhan99} for $p\geq1$ and
\cite[Lemma~6.4]{kold05} for general $p>0$.  Note that
(\ref{eq:u_fp=-langle-u}) is called the Blaschke-L\'evy representation
of the norm, which is discussed in detail in \cite{kold97}.  It should
be noted that $\exp\{-\|u\|_F^p\}$ is positive definite for
$p\in(0,2]$ if and only if $F\in\sL_p$, see \cite{kold92,kol96p} for a
survey of related results. By comparing (\ref{eq:u_fp=-langle-u}) with
Theorem~\ref{th:sam-taq} and symmetrising, if necessary, the measure
$\mu$, we see that $\sL_\alpha$ is exactly the family of associated
star bodies of $\sas$ laws. In the following we often switch between
the letters $\alpha$ and $p$, since the former is common in the
literature on stable laws, while the latter is typical in convex
geometry and functional analysis.

It is shown in \cite[Cor.~6.7]{kold05} that if $F$ is an $L_p$-ball
with $p\in(0,2]$, then $F$ is also an $L_r$-ball for each $r\in(0,p)$.
It is instructive to provide a probabilistic proof of this fact.

\begin{theorem}
  \label{th:lp-z-varying}
  If $F$ is an $L_p$-ball for $p\in(0,2]$, then $F$ is also an
  $L_r$-ball for all $r\in(0,p]$.
\end{theorem}
\begin{proof}
  Consider an $\sas$ random vector $\xi$ with $\alpha=p$ and the
  associated star body $F$. Let $\zeta$ be a non-negative stable
  random variable with $\beta\in(0,1)$.  Then the characteristic
  function of $\xi'=\zeta^{1/\alpha}\xi$ is given by
  \begin{displaymath}
    \E e^{i\langle \xi',u\rangle}=e^{-\|u\|_F^{\alpha\beta}}\,.
  \end{displaymath}
  Thus, $F$ is the associated star body of the symmetric stable $\xi'$
  with the characteristic exponent $\alpha\beta$, so that $F$ is an
  $L_r$-ball for $r=\alpha\beta=p\beta<p$. Note that $\xi$ and $\xi'$
  share the same associated star body $F$. 
\end{proof}

\section{Zonoids and $\sas$ laws with  $\alpha\in[1,2]$}
\label{sec:zono-assoc-sas}

If $\alpha\in[1,2]$, it is possible to arrive at a dual interpretation
of the characteristic function (\ref{eq:sam-taq}) by noticing that
$|\langle u,x\rangle|$ is the support function of the segment
$[-x,x]$, so that
\begin{equation}
  \label{eq:hz}
  \int_{\SS} |\langle u,z\rangle|^\alpha \sigma(dz)
  =\int_{\SS} h([-z,z],u)^\alpha \sigma(dz)=h(K,u)^\alpha\,,
\end{equation}
where 
\begin{displaymath}
  K=\sigma(\SS)^{1/\alpha}\E_\alpha [-\eta,\eta]
\end{displaymath}
is the rescaled Firey $\alpha$-expectation of the random set
$X=[-\eta,\eta]$ and $\eta$ is a random vector with values in $\SS$
distributed according to the normalised spectral measure $\sigma$.
The Minkowski inequality implies that $h(K,u)$ is indeed a support
function of a convex set.  If $\alpha=1$, then $K$ is called a zonoid,
see \cite[Sec.~3.5]{schn}.  Note that representation (\ref{eq:hz})
appears already in \cite{bret:dac:kriv66} and \cite{her63} in view of
its relationship to stable distributions and negative definite
functions on one hand and $L_p$-balls on the other one.

\begin{definition}
  \label{def:z}
  Let $\sigma$ be a finite measure on $\Sphere$.  A convex set $K$ in
  $\R^d$ is called \emph{$L_p$-zonoid} with $p\geq1$ and spectral
  measure $\sigma$ if $K=c^{1/p}\E_p[-\eta,\eta]$, where $c$ is the
  total mass of $\sigma$ and $\eta$ is distributed according to
  $c^{-1}\sigma$.
\end{definition}

It is obvious that $L_1$-zonoids are conventional zonoids.  If
$\sigma$ is a $p$-integrable probability measure on $\R^d$, then the
$L_p$-zonoid can be defined as $\E_p[-\eta,\eta]$ where $\eta$ has
distribution $\sigma$.  The following result now becomes an easy
corollary from Theorem~\ref{th:sam-taq}.

\begin{theorem}
  \label{thr:z-rep}
  A random vector $\xi$ is $\sas$ with $\alpha\in[1,2]$ if and only if
  there exists a unique centred $L_\alpha$-zonoid $K$ such that the
  characteristic function of $\xi$ is given by
  \begin{equation}
    \label{eq:z-rep}
    \phi_\xi(u)=e^{-h(K,u)^\alpha}\,,\quad u\in\R^d\,.
  \end{equation}  
\end{theorem}

The set $K$ from Theorem~\ref{thr:z-rep} is said to be the
\emph{associated zonoid} of $\xi$.  The corresponding polar set
$F=K^*$ is convex and becomes the associated star body of $\xi$.  It
is well known that all centred convex compact sets on the plane are
$L_1$-zonoids (i.e. classical zonoids), while this no longer holds in
dimensions 3 and more.  It follows immediately from
Theorem~\ref{th:lp-z-varying} that the family of $L_p$-zonoids becomes
richer if $p\in[1,2]$ decreases.

The following result provides a further interpretation of the
well-known fact saying that the exponentials of support functions of
zonoids are positive definite, see \cite[p.~194]{schn}.

\begin{cor}
  \label{cor:zc}
  The function $\phi(u)=e^{-h(K,u)^\alpha}$, $u\in\R^d$, with
  $\alpha\in[1,2]$ and a centred convex body $K\subset\R^d$ is a
  characteristic function if and only if $K$ is $L_\alpha$-zonoid.  In
  this case $\phi$ is necessarily the characteristic function of
  $\sas$ random vector.
\end{cor}

A measure $\sigma$ on $\Sphere$ is called \emph{isotropic} if the
$L_2$-zonoid with spectral measure $\sigma$ is a centred Euclidean
ball, see \cite{lut:yan:zhan04v,mil:paj89}. In other words, if an
isotropic $\sigma$ is taken as the spectral measure of a Gaussian random
vector, then this Gaussian vector has i.i.d.  coordinates.  The two
most common examples are the uniform measure on $\Sphere$ and the
cross measure having atoms of equal weights at $\pm e_i$ for the
canonical basis $e_1,\dots,e_d$.  
Note that the isotropy of $\sigma$ does not mean that the
corresponding $\sas$ vectors (with $\alpha$ not necessarily equal 2)
has a Euclidean ball as its associated star body.

\begin{example}[Independent/completely dependent components]
  \label{ex:i-d-zon}
  The components of $\sas$ vector $\xi$ with $\alpha\in[1,2]$ are
  independent if and only if its associated zonoid $K$ is a rescaled
  $\ell_\alpha$-ball, i.e. 
  \begin{displaymath}
    K=\{(a_1x_1,\dots,a_dx_d):\; x\in B_\alpha^d\}
  \end{displaymath}
  for $a_1,\dots,a_d\in\R$. If some of the $a_i$'s vanish, then $\xi$
  is no longer full-dimensional. Thus, an $\ell_q$-ball is
  $L_r$-zonoid for all $r\in[1,p]$ with $p$ being reciprocal to $q$.

  Furthermore, $\xi=(a_1\xi_1,\dots,a_d\xi_1)$ for
  $a=(a_1,\dots,a_d)\in\R^d$ and so has completely dependent
  components if and only if $K$ is the segment with end-points $\pm
  a$. In this case $\xi$ is not full-dimensional for each $a$ and
  $d\geq2$.
\end{example}

\begin{example}[Ellipsoids and sub-Gaussian laws]
  \label{ex:sub-gauss-law}
  The family of full-dimensional $L_2$-zonoids is the family of
  centred \emph{ellipsoids} in $\R^d$, that also correspond uniquely
  to non-degenerate Gaussian laws on $\R^d$. Thus ellipsoids are also
  $L_p$-zonoids for any $p\in[1,2]$.  Since polar sets to ellipsoids
  are again ellipsoids, ellipsoids are also $L_p$-balls for each
  $p\in(0,2]$.  Ellipsoids do not have a unique spectral measure for
  $\alpha=2$.  However, if an ellipsoid is represented as an
  $L_p$-zonoid with $p\in[1,2)$ or an $L_p$-ball with $p\in(0,2)$,
  then its spectral measure is unique. The corresponding $\sas$ random
  vector is said to have a \emph{sub-Gaussian} distribution, see
  \cite[Sec.~2.5]{sam:taq94}.

  An \emph{elliptical norm} is determined by a positive definite
  symmetric matrix $C$, so that $\|u\|_E=\langle Cu,u\rangle$ for the
  corresponding centred ellipsoid $E$. A simple quadratic optimisation
  argument yields that
  \begin{displaymath}
    h(E,u)=\sqrt{\langle C^{-1}u,u\rangle}\,,
  \end{displaymath}
  see, e.g., \cite{juh95}.
\end{example}

\begin{example}[Sub-stable laws]
  \label{ex:substable}
  The distribution of $\xi'$ from the proof of
  Theorem~\ref{th:lp-z-varying} is called \emph{sub-stable}.  If $\xi$
  is $\sas$ with $\alpha\in[1,2)$ and the associated zonoid $K$, then
  $\xi'$ is stable with the characteristic exponent
  $\alpha'=\alpha\beta$ and
  \begin{displaymath}
    \E e^{i\langle \xi',u\rangle}=e^{-h(K,u)^{\alpha\beta}}
    =e^{-\|u\|_{K^*}^{\alpha\beta}}\,.
  \end{displaymath}
  In this case the star body associated with $\xi'$ is convex and is
  equal to the polar set to $K$. In particularly, this holds for all
  sub-Gaussian distributions whose associated star bodies are
  ellipsoids for each $\alpha\in(0,2)$. 
\end{example}

\begin{theorem}
  \label{thr:strict-convex}
  Each $L_p$-zonoid with $p>1$ and spectral measure which is not
  concentrated on a great sub-sphere of $\Sphere$ is strictly convex,
  i.e. its support function is differentiable at every point.
\end{theorem}
\begin{proof}
  If $p>1$, then $|\langle u,v\rangle|^p$ is a differentiable function
  of $u$, so its integral is also differentiable. Since $\sigma$ is
  full-dimensional, the integral with respect to $\sigma$ does not
  vanish, so that its $\frac{1}{p}$th power is also differentiable.
  The equivalence of strict convexity and differentiability properties
  is explained in \cite[Cor.~1.7.3]{schn}.
\end{proof}

The strict convexity of $K$ means that for each $u\in\R^d$ the
\emph{support set}
\begin{equation}
  \label{eq:support-set}
  T(K,u)=\{y\in K:\; \langle y,u\rangle=h(K,u)\}
\end{equation}
is a singleton $\{x\}$ and the gradient of $h(K,u)$ equals $x$.
Theorem~\ref{thr:strict-convex} implies that polytopes cannot be
$L_p$-zonoids for $p>1$, so that the approximation by polytopes (often
used in the studies of zonoids) is no longer useful for $L_p$-zonoids
with $p>1$.

\section{Symmetric stable densities}
\label{sec:symm-stable-dens}

\subsection{Value of the density at the origin}
\label{sec:value-density-at}

Consider $\sas$ random vector $\xi$ with $\alpha\in(0,2]$ and the
characteristic function given by (\ref{eq:ch-f-rep}). It is useful to
interpret this characteristic function as
\begin{equation}
  \label{eq:phi_x-hk-ualph}
  \phi_\xi(u)=e^{-\uF^\alpha}=\Prob{\zeta\geq \uF}
  =\E \Ind{\zeta\geq \uF} =\E\Ind{u\in\zeta F}\,,
\end{equation}
where $\zeta$ is a non-negative random variable with $\Prob{\zeta\geq
  x}=e^{-x^\alpha}$ for $x>0$, so that
\begin{equation}
  \label{eq:moment}
  \E\zeta^\lambda=\Gamma(1+\lambda/\alpha)\,,\quad \lambda>-\alpha\,.
\end{equation}
The inversion formula for the Fourier transform yields the following
expression for the probability density function $f$ of $\xi$
\begin{align}
  (2\pi)^d f(x) =\int_{\Rd} e^{-i\langle
    u,x\rangle}\phi_\xi(u)du
  &=\E \int_{\Rd} e^{-i\langle u,x\rangle}\Ind{u\in\zeta F}
  du\nonumber \\
  \label{eq:chf}
  &=\E \int_{\zeta F} e^{i\langle u,x\rangle} du\,.
\end{align}
Note that we have used the fact that $F$ is centred. Since $f$ is the
expectation of the characteristic function of the uniform law on
$\zeta F$, the bounds on this characteristic function (see, e.g.,
\cite[Th.~1]{kul:prok02}) can be used to derive bounds for $f$. 

By substituting $x=0$ in (\ref{eq:chf}) we obtain
\begin{equation}
  \label{eq:f0=fr-gamm-}
  f(0)=\frac{1}{(2\pi)^d}\;\Gamma(1+\frac{d}{\alpha})|F|\,.
\end{equation}
Recall that the volumes of $F$ and its polar set $K=F^*$ (in case
$\alpha\geq1$) are related by the Blaschke-Santal\'o inequality as
\begin{displaymath}
  |F|\cdot |K|\leq \kappa_d^2
\end{displaymath}
with the equality if and only if $F$ is an ellipsoid, i.e. $\xi$ is
sub-Gaussian.  

If the spectral measure $\sigma$ is isotropic with the $L_2$-zonoid
being the unit Euclidean ball and $\alpha\geq1$, then it is possible
to apply the results from \cite{lut:yan:zhan04v} in order to bound the
volume of $F$ as
\begin{equation}
  \label{eq:omeg-fleq-omeg}
  \omega_d(2)/c_\alpha\leq |F|\leq \omega_d(\alpha)\,,
\end{equation}
where 
\begin{displaymath}
  \omega_d(\alpha)=2^d\; \frac{\Gamma(1+\frac{1}{\alpha})}
  {\Gamma(1+\frac{d}{\alpha})}\,,\quad
  c_\alpha^{\alpha/d}= \frac{\Gamma(1+\frac{d}{2})}{\Gamma(1+\thf)}\;
  \frac{\Gamma(\frac{1+\alpha}{2})}{\Gamma(\frac{d+\alpha}{2})}\,.
\end{displaymath}
If $\alpha\in[1,2)$, then the equality on the left in
(\ref{eq:omeg-fleq-omeg}) is achieved if $\sigma$ is a suitably
normalised Lebesgue measure on $\Sphere$, while the equality on the
right holds if $\sigma$ is concentrated on $\pm e_1,\dots,\pm e_d$.

\subsection{Derivatives at the origin}
\label{sec:deriv-at-orig}

Since $\phi_\xi(u)$ multiplied by a product of powers of the
coordinates of $u$ is integrable, representation (\ref{eq:chf})
implies that $f$ is infinitely differentiable. Its derivatives at the
origin are given by
\begin{displaymath}
  (2\pi)^d \frac{\partial^{2m} f}{\partial x_1^{k_1}\cdots\partial
    x_d^{k_d}}\bigg|_{x=0}
  =(-1)^m\Gamma(1+\frac{2m+d}{\alpha})
  \int_F v_1^{k_1}\cdots v_d^{k_d} dv\,,
\end{displaymath}
where $2m=k_1+\cdots+k_d$.  The central symmetry of $F$ implies that
the partial derivatives of odd orders vanish. By combining these
partial derivatives (with $m=1$) we arrive at the following expression
\begin{displaymath}
  \left(\sum_{i=1}^d w_i\frac{\partial }{\partial x_i}\right)^2 f\Big|_{x=0}
  =-\,\frac{1}{(2\pi)^d}\, \Gamma(1+\frac{2+d}{\alpha})
  \int_F \langle w,v\rangle^2 dv\,,
\end{displaymath}
where $w=(w_1,\dots,w_d)$.  The integral in the right-hand side can be
written as $\|w\|_E^2$ where $E$ is an ellipsoid in $\R^d$ called (for
a convex $F$) the \emph{Binet ellipsoid} of $F$.  This ellipsoid is
homothetic to the Legendre ellipsoid of $F$, which shares the moments
of inertia with $F$, see \cite{mil:paj89}.
Results from \cite{mil:paj89} can be used in order to bound the
integral of $\langle w,v\rangle^2$ over $F$.

Note that
\begin{displaymath}
  \sum_{i=1}^d w_i^2\frac{\partial^2 f}{\partial x_i^2}\Big|_{x=0}
  =-\frac{1}{(2\pi)^d}\Gamma(1+\frac{d+2}{\alpha})\int_F\sum_{i=1}^d
  w_i^2v_i^2 dv\,,
\end{displaymath}
where $(w_1^2 x_1^2+\cdots+w_d^2 x_d^2)$ defines an elliptic norm of
$x$ with the unit ball being the centred ellipsoid $E$ with semi-axes
$w_1^{-1},\dots,w_d^{-1}$.  Corollary~2.2a of \cite{mil:paj89} yields
that
\begin{displaymath}
  \int_F (w_1^{2}x_1^2+\cdots+w_d^{2}x_d^2) dx
  \geq \frac{d}{d+2} |F|^{1+\frac{2}{d}}
  \left(w_1\cdots w_d\right)^{2/d}\kappa_d^{-2/d}\,,
\end{displaymath}
and so provides an upper bound for the weighted sum of the second
derivatives of the density $f$ of $\xi$ at the origin as
\begin{displaymath}
  \sum_{i=1}^d w_i^2\frac{\partial^2 f}{\partial x_i^2}\Big|_{x=0}
  \leq -\;\frac{4\pi d}{d+2}
  \frac{\Gamma(1+\frac{d+2}{\alpha})\Gamma(1+\frac{d}{2})^{2/d}}
  {\Gamma(1+\frac{d}{\alpha})^{1+2/d}}(w_1\cdots w_d)^{2/d}
  f(0)^{1+2/d}\,,
\end{displaymath}
with the equality attained if $F$ is a dilate of the ellipsoid $E_w$,
i.e. for the corresponding sub-Gaussian law.

\subsection{Expectation of integrable functions}
\label{sec:expect-integr-funct}

Integrating (\ref{eq:chf}) leads to the following expression
\begin{equation}
  \label{eq:e-gxi=fr-eleft}
  \E g(\xi)=\frac{1}{(2\pi)^d} \E\left[\int_{\zeta F} \hat{g}(-v)dv\right]\,,
\end{equation}
where $\hat{g}$ is the Fourier transform of an integrable function
$g$. If $g$ is the Fourier transform of a measure $\mu$, then $\E
g(\xi)=\E\mu(\zeta F)$.  For example, $\E\exp\{-\|\xi\|^2/2\}$
equals the expected standard Gaussian content of $\zeta F$. If $g$ is
the indicator of the Euclidean ball $B_r$ of radius $r$ centred at the
origin, then
\begin{displaymath}
  \hat g(u)=(2r\pi/\|u\|)^{d/2} J_{d/2}(r\|u\|)\,,
\end{displaymath}
where $J_{d/2}$ is the Bessel function. Therefore
\begin{align*}
  \Prob{\|\xi\|\leq r}
  &=\left(\frac{r}{2\pi}\right)^{d/2}
  \int_F \|v\|^{-d/2} \E[\zeta^{d/2}J_{d/2}(r\zeta \|v\|)]dv\,.
\end{align*}

It is also possible to choose $g(\xi)$ to be the product of functions of
individual coordinates of $\xi$, i.e. 
\begin{displaymath}
  g(\xi)=\prod_{i=1}^d g_i(\xi_i)\,.
\end{displaymath}
For instance, if $g_i(x_i)=\Ind{[-a_i,a_i]}(x_i)$, $i=1,\dots,d$, then
\begin{align*}
  \E \left[\prod_{i=1}^d g_i(\xi_i)\right]=
  \Prob{\xi\in\times_{i=1}^d [-a_i,a_i]}
  =\pi^{-d} \int_F \E\prod_{i=1}^d \frac{\sin(a_iv_i\zeta)}{v_i}dv\,.
\end{align*}
The same argument with the Laplace density
$g_i(x_i)=\frac{\lambda_i}{2}e^{-\lambda_i|x_i|}$ yields that
\begin{displaymath}
  \E \exp\left\{-\sum \lambda_i|\xi_i|\right\}
  =\pi^{-d}\int_F \E\left[\prod_{i=1}^d
    \frac{\zeta\lambda_i}{\zeta^2v_i^2+\lambda_i^2}\right]dv\,.
\end{displaymath}
Note that in all these cases the dependency structure is expressed by
the set $F$ which determines the integration domain, while the value
of $\alpha$ influences the integrand which is the expectation of a
certain function of $\zeta$.

\subsection{R\'enyi entropy and related integrals}
\label{sec:renyi-entr-relat}

Another instance of (\ref{eq:e-gxi=fr-eleft}) appears if $g$ is itself
the density of $\saps$ law with associated star body $F'$. Then
\begin{equation}
  \label{eq:e-gxi=fr-zeta}
  \E g(\xi)=\frac{1}{(2\pi)^d}\E |\zeta F\cap \zeta' F'|\,,
\end{equation}
where $\Prob{\zeta'>x}=e^{-x^{\alpha'}}$ and $\zeta'$ is independent
of $\zeta$. 

\begin{theorem}
  \label{thr:int-dens}
  If $\xi$ is $\sas$ with associated star body $F$ and
  $\alpha\in(0,2]$, then, for all $c\neq0$, the density $f$ of $\xi$
  satisfies
  \begin{displaymath}
    \int_{\R^d} f(cx)f(x)dx = (1+c^\alpha)^{-d/\alpha} f(0)\,.
  \end{displaymath}
\end{theorem}
\begin{proof}
  Apply (\ref{eq:e-gxi=fr-zeta}) with $g(x)=c^d f(cx)$ and
  $\alpha=\alpha'$, so that $F'=cF$. Then
  \begin{align*}
    \E g(\xi)=c^d\int_{\Rd} f(x)f(cx)dx
    &=\frac{1}{(2\pi)^d}|F|\E(\min(\zeta,c\zeta'))^d\\
    &=\frac{1}{(2\pi)^d}\frac{\Gamma(1+\frac{d}{\alpha})}
    {(1+c^{-\alpha})^{d/\alpha}}\;|F|\,.
  \end{align*}
  Then note that $|F|$ is related to $f(0)$ by
  (\ref{eq:f0=fr-gamm-}). 
\end{proof}

By choosing $c=1$ we see that the density of each $\sas$ law satisfies
\begin{displaymath}
  \int_{\R^d} f(x)^2dx=2^{-d/\alpha} f(0)\,.
\end{displaymath}
The left-hand side can be recognised as the inverse to the $2$-R\'enyi
entropy power of $\xi$.

\subsection{Probability metric and distance to sub-Gaussian law}
\label{sec:prob-metr-dist}

Some useful \emph{probability metrics} are defined using logarithms of
characteristic functions of random vectors. Extending the definition
of the distance between two random variables from
\cite[Ex.~I.1.15]{zol} for the multivariate case, it is possible to
calculate the distance between two $\sas$ vectors $\xi'$ and $\xi''$
with the same characteristic exponent $\alpha\in[1,2]$ and the
associated zonoids $K_1$ and $K_2$ as
\begin{align*}
  \mathfrak{m}_\alpha(\xi',\xi'')
  &=\sup\{\|u\|^{-\alpha}|\log \E e^{i\langle
    u,\xi'\rangle}-\log \E e^{i\langle
    u,\xi''\rangle}|:\; u\in \R^d\}\\
  &=\sup\{|h(K_1,u)^\alpha-h(K_2,u)^\alpha|:\; u\in\Sphere\}\,.
\end{align*}
If $\alpha=1$, the right-hand side becomes the Hausdorff distance
between $K_1$ and $K_2$.

\begin{theorem}
  \label{thr:sg-approx}
  For each $\sas$ vector $\xi$ with $\alpha\in[1,2)$ in $\R^d$ and the
  associated zonoid $K$ there exists a sub-Gaussian $\sas$ vector
  $\eta$ such that $\mathfrak{m}_\alpha(\xi,\eta)\leq
  (d^{\alpha/2}-1)\|K\|^\alpha$.
\end{theorem}
\begin{proof}
  For each centred convex body $K$ in $\R^d$ there exists a centred
  ellipsoid $E$ (called the \emph{John ellipsoid}) such that $E\subset
  K\subset\sqrt{d}E$, see e.g. \cite[Th.~4.2.12]{gard95}. Then it
  suffices to note that
  \begin{displaymath}
    |h(K,u)^\alpha-h(E,u)^\alpha|
    \leq |d^{\alpha/2}h(E,u)^\alpha-h(E,u)^\alpha|
    \leq h(E,u)^\alpha(d^{\alpha/2}-1)
  \end{displaymath}
  and use the fact that $h(E,u)\leq\|E\|\leq\|K\|$ for all
  $u\in\Sphere$.
\end{proof}

Results from Section~\ref{sec:moments-sas-laws} can be used to related
moments of $\xi$ and $\eta$. For instance, Theorem~\ref{thr:norm}
implies that $\E\|\xi\|^\lambda/\E\|\eta\|^\lambda\in
[1,d^{\lambda/2}]$.

\section{Homogeneous functions of $\sas$ laws}
\label{sec:moments-sas-laws}

\subsection{Moments of the norm}
\label{sec:moments-norm}

If $g$ is a homogeneous function, i.e. $g(cx)=c^\lambda g(x)$ for all
$x\in\R^d$ and $c>0$, and so is not integrable over $\R^d$, then one
can interpret its Fourier transform using generalised functions. We
refer to \cite{gel:shil64} for the thorough account of generalised
functions and their Fourier transforms.  The left-hand side of
(\ref{eq:e-gxi=fr-eleft}) for not necessarily integrable $g$ can be
interpreted as the action of $g$ on the density $f$ of $\xi$ (denoted
$(g,f)$), while the right-hand side as the action of the Fourier
transform $\hat{g}$ of $g$ on $\phi_\xi$, i.e.
\begin{displaymath}
  (g,f)=\frac{1}{(2\pi)^d} (\hat g,\phi_\xi)
\end{displaymath}
is Parseval's identity. Since $\phi_\xi$ given by
(\ref{eq:phi_x-hk-ualph}) is not necessarily infinitely differentiable,
the action of a generalised function on it should be interpreted as
limits if the action of $\hat g$ does not involve differentiation.

\begin{theorem}
  \label{thr:norm}
  If $\xi$ is $\sas$ and $\lambda\in(-d,\alpha)$, then
  \begin{equation}
    \label{eq:exil-=-lambd}
    \E\|\xi\|^\lambda = \frac{2^{\lambda-1}}{\pi^{d/2}}
    \Gamma(\frac{d+\lambda}{2})
    \frac{\Gamma(1-\frac{\lambda}{\alpha})}{\Gamma(1-\thaf{\lambda})}
    \int_{\Sphere} \|u\|_F^\lambda\, du\,.
  \end{equation}
\end{theorem}
\begin{proof}
  Consider (\ref{eq:e-gxi=fr-eleft}) for
  $g(x)=\|x\|^\lambda=r^\lambda$.  Using the expression for the
  Fourier transform of $g$ (see \cite[Sec.~II.3.3]{gel:shil64}) one
  arrives at
  \begin{displaymath}
    \E\|\xi\|^\lambda=\frac{2^\lambda}{\pi^{d/2}}\;
    \frac{\Gf{d+\lambda}}{\Gamma(-\,\frac{\lambda}{2})}
    \Big( r^{-\lambda-d},\E\Ind{\zeta F}\Big)\,,
  \end{displaymath}
  where $(r^{-\lambda-d},\psi)$ denotes the action of the generalised
  function $r^{-\lambda-d}$ on the test function $\psi$.  If
  $0<\lambda<\alpha$, then it is possible to use the regularisation
  for $r^{-\lambda-d}$ (see \cite[Sec.~I.3.9]{gel:shil64}) to obtain
  \begin{displaymath}
    \Big(r^{-\lambda-d},\E\Ind{\zeta F}\Big)
    =\E(\zeta^{-\lambda}) \omega_d \int_0^\infty t^{-\lambda-1}(S_F(t)-1)dt\,,
  \end{displaymath}
  where $\omega_d$ is the surface area of the unit sphere in $\R^d$
  and $S_F(t)$ is the ratio of the surface areas of $S(t)\cap F$ and
  the sphere $S(t)$ of radius $t$. Then 
  \begin{align*}
    \omega_d \int_0^\infty t^{-\lambda-1}(1-S_F(t))dt
    &=\int_{\Sphere} \int_0^\infty t^{-\lambda-1}(1-\Ind{ut\in F})dtdu\\
    &=\int_{\Sphere} \int_{\|u\|_F^{-1}}^\infty t^{-\lambda-1}dtdu\,,
  \end{align*}
  which, together with the expression (\ref{eq:moment}) for the moment
  of $\zeta$, proves (\ref{eq:exil-=-lambd}) for $\lambda>0$.

  If $\lambda\in(-d,0)$, then no regularisation is needed, so that
  \begin{displaymath}
    \E\|\xi\|^{\lambda}=\frac{2^\lambda}{\pi^{d/2}}\;
    \frac{\Gamma(\frac{d+\lambda}{2})}{\Gamma(-\,\frac{\lambda}{2})}\;
    \Gamma(1-\frac{\lambda}{\alpha})\int_F \|u\|^{-\lambda-d}du\,.
  \end{displaymath}
  Then (\ref{eq:exil-=-lambd}) is obtained by passing to polar
  coordinates and using the fact that
  $\Gamma(1-\lambda/2)=(-\lambda/2)\Gamma(-\lambda/2)$. A direct check
  shows that (\ref{eq:exil-=-lambd}) holds also for $\lambda=0$.
\end{proof}

\begin{remark}
  \label{rem:wave}
  An alternative proof of Theorem~\ref{thr:norm} can be carried over
  using the \emph{plane-wave expansion} of the Euclidean norm
  \begin{displaymath}
    \|x\|^\lambda=\frac{1}{2\pi^{(d-1)/2}}\;
    \frac{\Gamma(\frac{d+\lambda}{2})}{\Gamma(\frac{1+\lambda}{2})}
    \int_{\Sphere} |\langle u,x\rangle|^\lambda du\,,
  \end{displaymath}
  see \cite[Sec.~3.10]{gel:shil64} and the expression for the
  moments of $|\langle u,\xi\rangle|$, see Theorem~\ref{thr:sc-prod}. 
\end{remark}

\begin{example}[Isotropic law]
  \label{ex:isot-mom}
  Assume that $\xi$ is \emph{isotropic}, i.e. $\|u\|_F=\sigma\|u\|$
  for all $u$ and $F=B_{\sigma^{-1}}$ is the ball of radius
  $\sigma^{-1}$.  Then (\ref{eq:exil-=-lambd}) and the expression for
  the surface area $\omega_d$ of the unit sphere imply that
  \begin{displaymath}
    \E\|\xi\|^\lambda=(2\sigma)^\lambda\;
    \frac{\Gamma(\frac{d+\lambda}{2})}{\Gamma(\frac{d}{2})}\;
    \frac{\Gamma(1-\frac{\lambda}{\alpha})}{\Gamma(1-\frac{\lambda}{2})}\,,
    \quad \lambda\in(-d,\alpha)\,,
  \end{displaymath}
  which is a well-known formula, see, e.g.,
  \cite[Eq.~(7.5.9)]{uch:zol99}.
\end{example}

If it is difficult to integrate $\|u\|_F^\lambda$ over the unit
sphere, it is possible to use trivial bounds $R^{-1}\leq \|u\|_F\leq
r^{-1}$, where $R$ and $r$ are the radii of the circumscribed and the
inscribed balls to $F$. The following lower bound is sharper for
$\lambda>0$.

\begin{cor}
  \label{cor:bound-mom}
  In the setting of Theorem~\ref{thr:norm} with
  $\lambda\in(0,\alpha)$, we have
  \begin{equation}
    \label{eq:exil-2lambda-fracg}
    \E\|\xi\|^\lambda\geq 2^\lambda\;
    \frac{\Gamma(\frac{d+\lambda}{2})}{\Gamma(\frac{d}{2})}\;
    \frac{\Gamma(1-\frac{\lambda}{\alpha})}{\Gamma(1-\frac{\lambda}{2})}\;
    \left(\frac{\kappa_d}{|F|}\right)^{\lambda/d}
  \end{equation}
  with the equality if and only if $F$ is a Euclidean ball.
\end{cor}
\begin{proof}
  The expression 
  \begin{displaymath}
    \tilde{V}_{-\lambda}(L,F)=\frac{1}{d}\int_{\Sphere}
    \|u\|_L^{-d-\lambda}\|u\|_F^\lambda du
  \end{displaymath}
  is called the \emph{dual mixed volume} of star bodies $K$ and $L$
  (note that the original definition \cite{lut96} is written for the
  radial functions of $L$ and $F$). Now it suffices to apply the dual
  mixed volume inequality (see \cite{lut96} and
  \cite[(2.4)]{lut:yan:zhan04})
  \begin{displaymath}
    \tilde{V}_{-\lambda}(L,F)^d\geq |L|^{d+\lambda}|F|^{-\lambda}
  \end{displaymath}
  with $L$ being the unit Euclidean ball.
\end{proof}

Note that the right-hand side of (\ref{eq:exil-2lambda-fracg}) equals
$\E\|\eta\|^\lambda$, where $\eta$ is an isotropic $\sas$ random
vector with the associated star body being the Euclidean ball of the
same volume as $F$.

\begin{example}[Multivariate normal and sub-Gaussian distributions]
  \label{ex:norm-mom}
  If $\xi$ has a multivariate normal distribution with covariance
  matrix $C$, then $F$ is the ellipsoid $E$ with $\|u\|_E^2=\thf
  \langle Cu,u\rangle$ and (\ref{eq:exil-=-lambd}) implies
  \begin{equation}
    \label{eq:exil-1pid2g-int_sph}
    \E\|\xi\|^\lambda=\frac{2^{\lambda/2-1}}{\pi^{d/2}}\Gamma(\frac{d+\lambda}{2}) 
    \int_{\Sphere}\langle Cu,u\rangle^{\lambda/2} du
  \end{equation}
  for $\lambda\in(-d,2)$. By passing to the limit, the formula holds
  also for $\lambda=2$. The integral retains its value for $\xi$
  having a sub-Gaussian distribution with the same associated star
  body $F$. Thus, the ratio of the moments of the norm for a normal
  vector and the corresponding $\sas$ sub-Gaussian vector depends only
  on $\alpha$, dimension and the order of the moment.

  If $\sigma_1^2,\dots,\sigma_d^2$ are the eigenvalues of $C$, then
  $F=E$ has semi-axes $\sqrt{2}/\sigma_i$, $i=1,\dots,d$, whence
  $\kappa_d/|F|$ equals $2^{-d/2}\prod \sigma_i$ and
  Corollary~\ref{cor:bound-mom} yields that 
  \begin{equation}
    \label{eq:exil-2lambd-fracg}
    \E\|\xi\|^\lambda\geq 2^{\lambda/2}\;
    \frac{\Gamma(\frac{d+\lambda}{2})}{\Gamma(\frac{d}{2})}\;
    \frac{\Gamma(1-\frac{\lambda}{\alpha})}{\Gamma(1-\frac{\lambda}{2})}\;
    \left(\prod_{i=1}^d\sigma_i\right)^{\lambda/d}
  \end{equation}
  with the equality if and only if $\sigma_1=\cdots=\sigma_d$. In particular, if
  $C$ is diagonal, then 
  \begin{displaymath}
    \E (\xi_1^2+\cdots+\xi_d^2)^{\lambda/2}
    \geq \frac{\Gamma(\frac{d+\lambda}{2})}{\Gamma(\frac{d}{2})}\;
    \frac{\Gamma(\frac{1}{2})}{\Gamma(\frac{1+\lambda}{2})}\;
    \prod_{i=1}^d\E|\xi_i|^\lambda
  \end{displaymath}
  with the equality if and only if $\E|\xi_i|^\lambda$ does not depend
  on $i$.  For this, we have used (\ref{eq:exil-2lambd-fracg}) and the
  fact that
  \begin{displaymath}
    \E|\xi_i|^\lambda=2^{\lambda/2}
    \frac{\Gamma(\frac{1+\lambda}{2})}{\Gamma(\frac{1}{2})}\;
    \frac{\Gamma(1-\frac{\lambda}{\alpha})}{\Gamma(1-\frac{\lambda}{2})}\;
    \sigma_i^\lambda\,.
  \end{displaymath}
\end{example}

\begin{example}[$\sas$ vectors with i.i.d. components]
  Let $\xi$ be $\sas$ with the associated star body being
  $\ell_\alpha$-ball with $\alpha\in(0,2]$, so that its coordinates
  $\xi_1,\dots,\xi_d$ are i.i.d. $\sas$ random variables, see
  Example~\ref{ex:indep}. The formula for the volume of the
  $\ell_p$-ball from \cite[p.~11]{pis89} and
  (\ref{eq:exil-2lambda-fracg}) imply that
  \begin{displaymath}
    \E(\xi_1^2+\cdots+\xi_d^2)^{\lambda/2}
    \geq 2^{-\lambda}
    \frac{\Gamma(\frac{d+\lambda}{2})}{\Gamma(\frac{1+\lambda}{2})}\;
    \frac{\Gamma(\frac{1}{2})}{\Gamma(\frac{d}{2})}\;
    \frac{\Gamma(1+\frac{d}{\alpha})^{\lambda/d}}{\Gamma(1+\frac{1}{\alpha})^\lambda}\;
    \kappa_d^{\lambda/d} \E|\xi_1|^\lambda
  \end{displaymath}
  for all $\lambda\in(0,\alpha)$.  For the opposite inequality note
  that the largest Euclidean ball inscribed in $B_\alpha^d$ has radius
  $d^{\thf-\frac{1}{\alpha}}$, whence
  \begin{displaymath}
    \E(\xi_1^2+\cdots+\xi_d^2)^{\lambda/2}
    \leq \frac{1}{\sqrt{\pi}\Gamma(\frac{d}{2})}
    \frac{\Gamma(\frac{d+\lambda}{2})}{\Gamma(\frac{1+\lambda}{2})}\;
    d^{\frac{1}{\alpha}-\thf}\E|\xi_1|^\lambda
  \end{displaymath}
  for $\lambda\in(0,\alpha)$. The both inequalities turn into an
  equality for $\alpha=2$ and any $\lambda\in(0,2]$ and yield the
  well-known moments of the chi-square distribution with $d$ degrees
  of freedom.
\end{example}

Using bounds for the average values of norms on the unit sphere from
\cite{lit:mil:sch98}, it is possible to relate moments of different
orders.

\begin{cor}
  \label{cor:litvak}
  Let $\xi$ be $\sas$ with $\alpha\in(1,2]$ and the associated star
  body $F$.  Let $b$ be the radius of the largest centred Euclidean
  ball inscribed in $F$. Then for all $\lambda\in[1,\alpha)$
  \begin{displaymath}
    a_\lambda\max\left(M_1,\frac{c_1 b\sqrt{\lambda}}{\sqrt{d}}\right)^\lambda
    \leq \E\|\xi\|^\lambda
    \leq a_\lambda\max\left(2M_1,\frac{c_2 b\sqrt{\lambda}}{\sqrt{d}}\right)^\lambda\,,
  \end{displaymath}
  where $c_1$ and $c_2$ are absolute constants,
  \begin{displaymath}
    M_1 = \frac{\pi^{(d+1)/2}}{\Gamma(\frac{d+1}{2})\Gamma(1-\frac{1}{\alpha})}
    \E\|\xi\|= \int_{\Sphere}\|u\|_F du\,.
  \end{displaymath}
  and 
  \begin{displaymath}
    a_\lambda=\frac{2^{\lambda-1}\Gf{d+\lambda}\Gamma(1-\frac{\lambda}{\alpha})}
      {\pi^{d/2}\Gamma(1-\frac{\lambda}{2})}\,.
  \end{displaymath}
\end{cor}

Using \cite[Lemma~3.6]{kold05} for the Fourier transform of the
power of the $\ell_p$-norm $\|x\|_p^\lambda$ it is possible to arrive
at the following expression
\begin{displaymath}
  \E\|\xi\|_p^\lambda=\frac{1}{(2\pi)^d}\;
  \frac{p\Gamma(1-\frac{\lambda}{\alpha})}
  {\Gamma(-\,\frac{\lambda}{p})}
  \int_F\int_0^\infty s^{d+\lambda-1}\prod_{i=1}^d \gamma_p(sv_i) ds dv\,,
\end{displaymath}
where $\gamma_p$ is the Fourier transform of the function
$e^{-|x|^p}$, $x\in\R$. It is valid for $\lambda\in(-d,0)$ and for
$\lambda\in(0,\min(\alpha,dp))$ with non-integer $\lambda/p$.

\medskip

Although $\|\xi\|^\lambda$ is not necessarily integrable, the integral
in the right-hand side of (\ref{eq:exil-=-lambd}) is well defined for
all $\lambda>0$. The following result describes the limiting behaviour
of the $\lambda$-moment of $\|\xi\|$ as $\lambda\uparrow\alpha$.

\begin{cor}
  \label{cor:alpha}
  If $\xi$ is $\sas$ with $\alpha\in(0,2)$ and spectral measure $\sigma$, then 
  \begin{displaymath}
    \lim_{\lambda\uparrow\alpha}
    \frac{\E\|\xi\|^\lambda}{\Gamma(1-\frac{\lambda}{\alpha})}
    =\frac{2^\alpha}{\sqrt{\pi}}\;
    \frac{\Gamma(\frac{d+\alpha}{2})\Gamma(\frac{\alpha+1}{2})\Gamma(\frac{d}{2})}
    {\Gamma(1-\frac{\alpha}{2})\Gamma(\frac{d-1}{2})\Gamma(\frac{d+\alpha+1}{2})}
    \; \sigma(\Sphere)\,.
  \end{displaymath}
\end{cor}
\begin{proof}
  It suffices to refer to (\ref{eq:exil-=-lambd}) together with
  \begin{displaymath}
    \int_{\Sphere} \|u\|_F^\alpha du=\int_{\Sphere}\int_{\Sphere}
    |\langle u,y\rangle|^\alpha \sigma(dy)du
    =\int_{\Sphere} \left(\int_{\Sphere} |\langle u,y\rangle|^\alpha
      du\right)
    \sigma(dy)\,,
  \end{displaymath}
  and use the fact that
  \begin{displaymath}
    \int_{\Sphere} |\langle u,y\rangle|^\alpha du=
    \frac{2\pi^{(d-1)/2}}{\Gamma(\frac{d-1}{2})}\;
    \mathrm{B}(\frac{\alpha+1}{2},\frac{d}{2})\|y\|^\alpha\,,
  \end{displaymath}
  where $\mathrm{B}$ is the beta-function.
\end{proof}

\medskip

The \emph{intersection body} of a centred star body $L$ is the star
body $\I L$ such that
\begin{displaymath}
  \|u\|_{\I L}^{-1}=\Vol_{d-1}(L\cap u^\perp)\,,\quad u\in\Sphere\,.
\end{displaymath}
For $\xi\in\R^d$, define $\|\xi\|_{\I L}^{-1}=\|\xi\|\Vol_{d-1}(L\cap
\xi^\perp)$, where $\xi^\perp$ is the $(d-1)$-dimensional subspace
orthogonal to $\xi$.

\begin{theorem}
  \label{thr:i-b-norm}
  If $\xi$ is $\sas$ with associated star body $F$ and $d\geq2$, then
  \begin{displaymath}
    \E \|\xi\|_{\I F}^{-1}=\frac{1}{\pi(d-1)}\Gamma(1+\frac{1}{\alpha})|F|\,.
  \end{displaymath}
\end{theorem}
\begin{proof}
  It is known \cite[p.~72]{kold05} that the Fourier transform of
  $g(x)=\|x\|_{\I L}^{-1}$ for a star body $L$ is given by
  $(2\pi)^d\|u\|_L^{-d+1}/(\pi(d-1))$. Thus,
  \begin{align*}
    \E \|\xi\|_{\I L}^{-1} &= \frac{1}{\pi(d-1)}\E \int_{\zeta F}
    \|x\|_L^{-d+1}dx\\
    &=\frac{1}{\pi(d-1)}\Gamma(1+\frac{1}{\alpha})
    \int_{\Sphere}\rho_F(u)\rho_L(u)^{d-1} du\,.
  \end{align*}
  If $F=L$, the integral becomes the polar coordinate representation
  of the volume of $F$.
\end{proof}

Similarly to Theorem~\ref{thr:i-b-norm} and using
\cite[Th.~4.6]{kold05} it is possible to deduce that
\begin{displaymath}
  \E\|\xi\|_{\I_kL}^{-k}
  =\frac{1}{(2\pi)^k(d-k)}\Gamma(1+\frac{k}{\alpha})
  \int_{\Sphere}\rho_L(u)^{d-k}\rho_F(u)^k du\,,
\end{displaymath}
where $I_kL$ is the $k$-intersection body of $L$, so that this moment
is proportional to the volume of $F$ is $L=F$. Note that these
intersection bodies are defined from $\Vol_k(I_kL\cap
H^\perp)=\Vol_{n-k}(L\cap H)$ for each $(n-k)$-dimensional subspace
$H$, which differs by a factor of $2$ from the definition of $\I L$ for
$k=1$.

\subsection{Mixed moments}
\label{sec:joint-moments}

The following result deals with joint moments of the coordinates of
$\xi$. For a function $g(x_1,\dots,x_d)$ and $j=1,\dots,d$ denote
\begin{displaymath}
  \Delta_j g(x)=g(x)-g(x|_j)\,,
\end{displaymath}
where $x|_j$ is $x$ with the $j$th coordinate replaced by zero. 

\begin{theorem}
  \label{thr:joint-moments}
  If $\xi$ is $\sas$ and $\lambda_1,\dots,\lambda_d$ are positive
  numbers with $\lambda=\sum \lambda_i<\alpha$, then
  \begin{multline}
    \label{eq:jm}
    \E (|\xi_1|^{\lambda_1}\cdots |\xi_d|^{\lambda_d})
    =2^{\lambda-d}\frac{(-1)^d}{\pi^{d/2}}\Gamma(1-\frac{\lambda}{\alpha})
    \prod_{i=1}^d \frac{\lambda_i\Gamma(\frac{\lambda_i+1}{2})}
      {\Gamma(1-\frac{\lambda_i}{2})}\\
    \times\int_{\R^d} |u_1|^{-\lambda_1-1}\cdots |u_d|^{-\lambda_d-1} 
    (\Delta_1\cdots\Delta_d \Ind{F}(u))du\,.
  \end{multline}
\end{theorem}
\begin{proof}
  The result follows from the formula for the Fourier transform of
  $|x|^{\lambda}$ as
  $-2\sin(\lambda\pi/2)\Gamma(\lambda+1)|u|^{-\lambda-1}$ (see
  \cite[Sec.~II.2.3]{gel:shil64}) and the fact that the Fourier
  transform of the direct product $\prod |x_i|^{\lambda_i}$ is the
  direct product of Fourier transforms, see
  \cite[Sec.~II.3.2]{gel:shil64}. The expression
  $\Delta_1\cdots\Delta_d\Ind{F}(u)$ appears as a result of the
  regularisation procedure, see \cite[Sec.~I.3.2]{gel:shil64}.
  Finally, one needs the expression for the $(-\lambda)$th moment of
  $\zeta$ from (\ref{eq:moment}) and the fact that
  \begin{displaymath}
    \frac{1}{\pi}\sin \frac{\lambda\pi}{2}\Gamma(\lambda+1)
    =\frac{\lambda 2^{\lambda-1}}{\sqrt{\pi}}\;
    \frac{\Gamma(\frac{\lambda+1}{2})}{\Gamma(1-\frac{\lambda}{2})}\,.
  \end{displaymath}
\end{proof}

Since $\Delta_1\cdots\Delta_d \Ind{F}(u)$ vanishes in a neighbourhood
of the origin, the integral in (\ref{eq:jm}) is well defined.
If $d=2$, then (\ref{eq:jm}) turns into
\begin{multline*}
  \E (|\xi_1|^{\lambda_1}|\xi_2|^{\lambda_2})
  =\frac{2^{\lambda-2}}{\pi} \Gamma(1-\frac{\lambda}{\alpha})
  \prod_{i=1}^2\frac{\lambda_i\Gamma(\frac{\lambda_i+1}{2})}
      {\Gamma(1-\frac{\lambda_i}{2})}\\
  \times\int_{\R^2} |u_1|^{-\lambda_1-1}|u_2|^{-\lambda_2-1}
  \Big(\Ind{F}(u_1,u_2)-\Ind{F}(0,u_2)-\Ind{F}(u_1,0)+1\Big)du_1du_2\,.
\end{multline*}

The \emph{signed power} of a real number $t$ is defined by
\begin{equation}
  \label{eq:tspow-.-}
  t^\spower{\lambda}=|t|^\lambda\sign(t)\,,
\end{equation}
where $\sign(t)$ is the sign of $t$. 

\begin{theorem}
  \label{thr:s-p}
  If $\xi$ is $\sas$ in $\R^d$ for an even $d$ and
  $\lambda_1,\dots,\lambda_d$ are non-negative numbers, with none of
  them being 1 and such that $\lambda=\sum \lambda_i<\alpha$, then
  \begin{multline}
    \label{eq:s-p-moment}
    \E(\xi_1^\spower{\lambda_1}\cdots\xi_d^\spower{\lambda_d})
    =\frac{2^\lambda i^{d}}{\pi^{d/2}}
    \Gamma(1-\frac{\lambda}{\alpha})
    \prod_{i=1}^d\frac{\Gamma(1+\frac{\lambda_i}{2})}
    {\Gamma(\thf-\frac{\lambda_i}{2})}\\
    \times \int_F u_1^\spower{-\lambda_1-1}\cdots
    u_d^\spower{-\lambda_d-1} du\,,
  \end{multline}
  where the integral is understood as its principal value, i.e. the
  limit of the integral over $F\setminus \eps B$ as $\eps\to 0$.  The
  mixed moments vanish if $d$ is odd.
\end{theorem}
\begin{proof}
  Use the formula
  $2i\Gamma(\lambda+1)\cos(\lambda\pi/2)u^\spower{-\lambda-1}$ for the
  Fourier transform of the function $x^\spower{\lambda}$ with a
  non-integer $\lambda$, see \cite[Sec.~II.2.3]{gel:shil64} and
  identities for the Gamma function.
\end{proof}

For a centred star body $F$ in $\R^k$ denote
\begin{displaymath}
  \sI(F)=\int_F \frac{du}{u_1\cdots u_k}\,,
\end{displaymath}
where the integral is understood as its principal value (note that $F$
contains a neighbourhood of the origin in $\R^k$). Note that
$\sI(AF)=\sI(F)$ for each diagonal matrix $A$ with non-vanishing
entries and 
\begin{displaymath}
  \sI(F)=\int_{\Sphere[k-1]}\frac{\log \|v\|_F}{v_1\cdots v_k}dv\,.
\end{displaymath}

\begin{cor}
  \label{cor:zero-sign}
  If $\xi$ is $\sas$ random vector in $\R^d$ and $d$ is even, then
  \begin{equation}
    \label{eq:esignx-int_ffr-u_d}
    \E\sign(\xi_1\cdots\xi_d)=\frac{i^d}{\pi^d}\sI(F)\,.
  \end{equation}
\end{cor}

Since the left-hand side of (\ref{eq:esignx-int_ffr-u_d}) does not
exceed one in absolute value, we obtain an inequality $|\sI(F)|\leq
\pi^d$ valid for all centred star bodies $F\subset\R^d$.  Note that
the expectation in (\ref{eq:esignx-int_ffr-u_d}) does not depend on
$\alpha$.  If $d=2$, then
\begin{equation}
  \label{eq:e-signx-frac1p}
  \E \sign(\xi_1\xi_2)=-\frac{1}{\pi^2}\sI(F)\,.
\end{equation}

\medskip

Note that in (\ref{eq:s-p-moment}) at most one of the $\lambda_i$'s
equals $1$, since their total sum is strictly less than 2. The case of
$\lambda_i=1$ needs a special treatment, since the Fourier transform
of $x$ is given by $(-2\pi i)\delta'(u)$, i.e. it acts as $(-2\pi i)$
times the derivative of the test function at zero.  

\begin{theorem}
  \label{thr:one-s-p}
  Let $\xi$ be $\sas$ with $\alpha\in(1,2]$, the associated star body
  $F$ and the associated zonoid $K=F^*$.  If $d$ is even and
  $\lambda_2,\dots,\lambda_d$ are non-negative numbers such that
  $\lambda=1+\lambda_2+\cdots+\lambda_d<\alpha$, then
  \begin{multline}
    \E(\xi_1\xi_2^\spower{\lambda_2}\cdots\xi_d^\spower{\lambda_d})
    =-\;\frac{\alpha 2^{\lambda-1}i^{d}}{\pi^{(d-1)/2}}
    \Gamma(2-\frac{\lambda}{\alpha})
    \prod_{i=2}^d\frac{\Gamma(1+\frac{\lambda_i}{2})}
    {\Gamma(\thf-\frac{\lambda_i}{2})}\\
    \int_{F\cap e_1^\perp} u_2^\spower{-\lambda_2-1}\cdots
    u_d^\spower{-\lambda_d-1}\|u\|_F^{\alpha-1}
    h(T(K,u),e_1)du_2\cdots du_d\,,
  \end{multline}
  where $e_1=(1,0,\dots,0)$, $T(K,u)$ is the support set of $K$ in
  direction $u$, see (\ref{eq:support-set}), and the integral is
  understood as its principal value.  The mixed moments vanish if $d$
  is odd.
\end{theorem}
\begin{proof}
  Since $\alpha>1$, Theorem~\ref{thr:strict-convex} implies that the
  support function of $K$ is differentiable.  It is well known (see
  \cite[Th.~1.7.2]{schn}) that the directional derivative of the
  support function is given by
  \begin{equation}
    \label{eq:lim_sd-0-frachk}
    \lim_{s\downarrow 0} \frac{h(K,u+vs)-h(K,u)}{s}=h(T(K,u),v)\,.
  \end{equation}
  This formula for $v=e_1$ yields that the Fourier transform
  $\widehat{x_1}$ acts on $\phi(u)=e^{-h(K,u)^\alpha}$ as
  \begin{displaymath}
    (-2\pi i)e^{-h(K,u|_1)^\alpha}\alpha\|u|_1\|_F^{\alpha-1}
    h(T(K,u|_1),e_1)\,,
  \end{displaymath}
  where $u|_1=(0,u_2,\dots,u_d)$.  
  The remainder of the proof relies on the formulae for Fourier
  transforms of the signed powers as in Theorem~\ref{thr:s-p}.
\end{proof}

\medskip

The following result gives a formula for the probability that $\sas$
vector $\xi$ takes a value from a polyhedral cone.

\begin{theorem}
  \label{thr:plus}
  If $\xi$ is $\sas$ with associated star body $F$, then for each
  invertible matrix $A$ we have
  \begin{displaymath}
    \Prob{\xi\in A\R_+^d}=\frac{1}{(2\pi)^d}
    \sum_{m=0}^{\left[\thaf{d}\right]} \pi^{d-2m}(-1)^m
    \sum_{\{i_1,\dots,i_{2m}\}\subset\{1,\dots,d\}}
    \sI((A^\top F)\cap H_{i_1,\dots,i_{2m}})\,,
  \end{displaymath}
  where $H_{i_1,\dots,i_{2m}}$ is the hyperplane of dimension $2m$
  spanned by the basis vectors $e_{i_1},\dots,e_{i_{2m}}$.
\end{theorem}
\begin{proof}
  By \cite[II.2.3~(6)]{gel:shil64}, the Fourier transform of the
  generalised function $x_{j+}^0=\Ind{x_j\geq0}$ is given by
  $iu_j^{-1}+\pi\delta(u_j)$, so that the Fourier transform of
  $\Ind{x\in\R_+^d}$ is the product
  \begin{displaymath}
    \prod_{k=1}^d \left(\frac{i}{u_k}+\pi\delta(u_k)\right)\,.
  \end{displaymath}
  Now it suffices to open the parentheses in the product and use the
  fact that the delta function $\delta(u_k)$ applied to the indicator
  of $F$ yields $1$.

  Finally, it remains to note that $\Prob{\xi\in
    A\R_+^d}=\Prob{A^{-1}\xi\in\R_+^d}$ and that $A^{-1}\xi$ has the
  associated star body $A^\top F$.
\end{proof}

It is easy to see that the result of Theorem~\ref{thr:plus} corresponds
to (\ref{eq:e-signx-frac1p}) if $d=2$.  In a similar manner it is
possible to compute mixed moments of the positive parts of the
components of $\xi$.

\subsection{Integrals of the density}
\label{sec:integrals-density}

The following result expresses the integrals of the density over
1-dimensional subspaces of $\R^d$.

\begin{theorem}
  \label{thr:scalar}
  If $f$ is the density of $\sas$ law, then, for each unit vector $u$,
  \begin{equation}
    \label{eq:int_r-ftudt-=frac12p}
    \int_\R f(tu)dt
    =\frac{1}{(2\pi)^{d-1}}\Gamma(1+\frac{d-1}{\alpha})A_{F,u}\,,
  \end{equation}
  where $A_{F,u}=\Vol_{d-1}(F\cap u^\perp)$ is the $(d-1)$-dimensional
  Lebesgue measure of the intersection of $F$ with the subspace
  orthogonal to $u$.
\end{theorem}
\begin{proof}
  Using the technique of generalised functions, it is possible to
  calculate the Fourier transform of the function $g=\delta_{\langle
    u,x\rangle}$ for a fixed unit vector $u$ as $(\hat
  g,\psi)=(g,\hat\psi)$ for any test function $\psi$ and its Fourier
  transform $\hat\psi$, see \cite{gel:shil64}. A direct calculation
  shows that
  \begin{displaymath}
    (\hat g,\psi)=(2\pi)^{d-1}\int_\R \psi(tu)dt\,.
  \end{displaymath}
  By applying this expression to the density $f$ and using
  (\ref{eq:phi_x-hk-ualph}) we obtain that
  \begin{displaymath}
    (g,\phi)=\E \Vol_{d-1}((\zeta F)\cap u^\perp)
    =\Gamma(1+\frac{d-1}{\alpha})A_{F,u}\,.
  \end{displaymath}
\end{proof}

The question, if $A_{F_1,u}\leq A_{F_2,u}$ for convex sets $F_1$ and
$F_2$ and all $u\in\Sphere$ implies that the volume of $F_1$ is
smaller than the volume of $F_2$ is known in convex geometry under the
name of the \emph{Busemann--Petty problem}. This problem has been
recently completely solved (see, e.g. \cite{gar:kol:sch99} for the
solution based on the Fourier analysis) by establishing that the
answer is affirmative only in dimensions at most 4. The sets $F$ that
appear as associated star bodies of $\sas$ distributions are
$L_p$-balls and so are intersection bodies, for which the
Busemann--Petty problem has an affirmative answer in all dimensions,
see \cite[Sec.~4.3]{kold05}. In application to stable distributions
this means that if two $\sas$ densities $f_1$ and $f_2$ with the
same characteristic exponent satisfy
\begin{displaymath}
  \int_\R f_1(tu)dt\leq \int_\R f_2(tu)dt\,,\quad u\in\Sphere\,,
\end{displaymath}
then $f_1(0)\leq f_2(0)$. Recall that by (\ref{eq:f0=fr-gamm-}) the
value of the density at the origin is proportional to the volume of
$F$.

It is also possible to consider the intersection of $F$ with a
subspace $H_k$ of dimension $k$ and obtain that (see also
\cite[Lemma~3.24]{kold05})
\begin{displaymath}
  \int_{H_k^\perp} f(x)dx =\frac{1}{(2\pi)^k}
  \Gamma(1+\frac{k}{\alpha})\Vol_k(F\cap H_k)\,,
\end{displaymath}
which yields (\ref{eq:f0=fr-gamm-}) for $k=d$ and
(\ref{eq:int_r-ftudt-=frac12p}) for $k=d-1$. For $k=1$ we get
\begin{displaymath}
  \int_{\langle u,x\rangle=0} f(x) dx
  =\frac{1}{\pi}\Gamma(1+\frac{1}{\alpha})
  \rho_F(u)=\frac{1}{\pi}\Gamma(1+\frac{1}{\alpha})\|u\|_F^{-1}\,.
\end{displaymath}

It is also possible to express the integral of the type $\int_0^\infty
f(tu) t^{d+\lambda-1}dt$ by means of the action of the generalised
function $|t|^{-d-\lambda}$ on the test function
$A_{F,u}(t)=\Vol_{d-1}(F\cap (u^\perp+ tu))$. This yields the
$L_{d+\lambda}$-star of $\xi$, see \cite{lut:yan:zhan04}. In
particular, the $L_1$-star of $\xi$ has the radial function
(\ref{eq:int_r-ftudt-=frac12p}) and so is proportional to the
intersection body of $F$.

\subsection{Scalar products and zonoids of random vectors}
\label{sec:scal-prod-zono}

Moments of scalar products of $\xi$ with unit vector $u$ can be easily
calculated using the Fourier transform of the generalised function
$|\langle x,u\rangle|^\lambda$, see \cite[Lemma~3.14]{kold05}, or by
the explicit calculation of the moments of the $\sas$ random variable
$\langle\xi,u\rangle$.

\begin{theorem}
  \label{thr:sc-prod}
  If $\xi$ is $\sas$ and $u\in\Sphere$, then
  \begin{displaymath}
    \E |\langle\xi,u\rangle|^\lambda
    =2^\lambda \frac{\Gamma(\frac{\lambda+1}{2})}{\sqrt{\pi}}
    \;\frac{\Gamma(1-\frac{\lambda}{\alpha})}{\Gamma(1-\frac{\lambda}{2})}
    \|u\|_F^\lambda
  \end{displaymath}
  for $\lambda\in(-1,\alpha)$.
\end{theorem}

The \emph{zonoid} of an integrable random vector $\xi$ is defined as
the expectation of the random segment $X=[0,\xi]$, see
\cite{kos:mos98,mos02}. Representing $X$ as
$\thf\xi+[-\thf\xi,\thf\xi]$, the expectation of $X$ can be found from
\begin{displaymath}
  h(\E X,u)=\thf\langle \E\xi,u\rangle + \thf\E |\langle\xi,u\rangle|\,.
\end{displaymath}
If $\xi$ is $\sas$ with $\alpha\in(1,2]$, then $\E\xi=0$ and
Theorem~\ref{thr:sc-prod} with $\lambda=1$ yields that
\begin{displaymath}
  h(\E X,u)=\frac{1}{\pi}\Gamma(1-\frac{1}{\alpha})\|u\|_F\,,
\end{displaymath}
so that 
\begin{displaymath}
  \E X=\frac{1}{\pi}\Gamma(1-\frac{1}{\alpha})K\,.
\end{displaymath}
Thus, the zonoid of $\xi$ in the sense of \cite{mos02} coincides with
the rescaled associated zonoid of $\xi$. The volume of the zonoid $\E
X$ is closely related to the expectation of a random determinant whose
columns are i.i.d. realisations of $\xi$. Note also various
statistical applications of zonoids of random vectors, e.g. for
trimming of multivariate observations, see \cite{kos:mos97z}.
Furthermore, the associated zonoid of $\xi$ can be estimated as the
rescaled zonoid of $\xi$, e.g. by evaluating the Minkowski average of
$[0,\xi^{(i)}]$ for the i.i.d. sample $\xi^{(1)},\dots,\xi^{(n)}$.  In
order to recover the spectral measure from the associated zonoid, one
has to use the inversion formula for the $p$-cosine transform (see
\cite{kold97}) combined with a smoothing operation applied to the
support function of $K$.

\section{Stable laws in $\R_+^d$}
\label{sec:power-sums-maxima}

\subsection{Power sums}
\label{sec:power-sums-r_+d}

Stable laws with all non-negative components (or one-sided laws) are
traditionally called totally skewed to the right. However, if one
considers them on the semigroup $\R_+^d$ with addition, these
distributions can be also called symmetric stable laws, since this
semigroup has the identical involution, see \cite{dav:mol:zuy08}.
Still we retain the term $\sas$ only for origin-symmetric stable laws
in the whole space. The Laplace transform of one-sided strictly stable
law is given by
\begin{displaymath}
  \E e^{-\langle u,\xi\rangle}
  =\exp\left\{-\int_{\SSp} \langle u,y\rangle^\alpha
    \sigma(dy)\right\}\,,
  \quad u\in\R_+^d\,,
\end{displaymath}
where $\alpha\in(0,1)$ and the spectral measure $\sigma$ on
$\SSp=\Sphere\cap\R_+^d$ is unique. It is clearly possible to write
\begin{equation}
  \label{eq:l-f-one-s}
  \E e^{-\langle u,\xi\rangle}=e^{-\|u\|_F^\alpha}
\end{equation}
for a centred star-shaped (not necessarily convex) set $F$ from
(\ref{eq:u_k=int_ss-langle-u}) with the spectral measure obtained by
taking all possible reflections of $\sigma$ with respect to coordinate
planes.  

Below we show how to develop an alternative representation of the
Laplace exponent using convex sets.  For this purpose, it is useful to
work with generalised power sums.  For $p\in(0,\infty)$, the
\emph{$p$-sum} of two non-negative numbers $s$ and $t$ is defined by
\begin{displaymath}
  s\psum[p] t=(s^p+t^p)^{1/p}\,.
\end{displaymath}
If $p=\infty$, this operation turns into the maximum of $s$ and $t$.
The $p$-sum $x\psum[p] y$ for $x,y\in\R_+^d$ is defined
coordinatewisely as
\begin{displaymath}
  x\psum[p] y =(x_1\psum[p] y_1,\dots,x_d\psum[p] y_d)\,.
\end{displaymath}
  
Random vector $\xi$ in $\R_+^d$ is \emph{strictly stable for $p$-sums}
with characteristic exponent $\alpha\neq0$ if
\begin{equation}
  \label{eq:a1alph-b1alph-a+b1}
  a^{1/\alpha}\xi_1\psum[p] b^{1/\alpha}\xi_2\deq (a+b)^{1/\alpha}\xi
\end{equation}
for all $a,b>0$ and $\xi_1,\xi_2$ being independent copies of $\xi$.
The special cases correspond to the usual stability for
\emph{arithmetic} sums ($p=1$) and \emph{max-stability} ($p=\infty$).
The general results from \cite{dav:mol:zuy08} concerning stable
distributions on abelian semigroups imply that $\alpha\in(0,p]$.  It
is easy to see that $\xi$ satisfies (\ref{eq:a1alph-b1alph-a+b1}) with
$p\in(0,\infty)$ if and only if $\xi^p$ is strictly stable for
arithmetic sums with the characteristic exponent $\alpha'=\alpha/p$.
Note that a power of a vector is always understood coordinatewisely,
i.e.  $\xi^p=(\xi_1^p,\dots,\xi_d^p)$. The $p$th signed power of a set
$M\subset\R^d$ is defined as
\begin{equation}
  \label{eq:set-power}
  M^\spower{p}=\{x^\spower{p}:\; x\in M\}\,,
\end{equation}
where $p>0$ and $x^\spower{p}$ is the vector of signed powers of the
components of $x$, see (\ref{eq:tspow-.-}).

The analytical tools for $p$-sums rely on the concept of characters on
semigroups, see \cite{ber:c:r}. A \emph{character} $\chi$ is a
homomorphism between a semigroup and the unit complex disk with
multiplication operation.  The involution operation on the semigroup
corresponds to the complex conjugation operation on characters.  The
involution is identical if and only if all characters are real-valued.
In particular, in $\R_+$ with (arithmetic) addition (and identical
involution) the characters are $\chi(x)=e^{-tx}$; in $\R$ with
addition (so that the involution is the negation) we set
$\chi(x)=e^{itx}$; in $\R_+$ with the coordinatewise maximum (and
identical involution) the characters are $\chi(x)=\Ind{x\leq t}$ for
$t\geq0$.

If the characters separate all points, i.e. if $x_1\neq x_2$ implies
$\chi(x_1)\neq \chi(x_2)$ for some $\chi$ and the characters generate
the Borel $\sigma$-algebra, then the \emph{Laplace transform} $\chi
\mapsto \E\chi(\xi)$ characterises uniquely the distribution of a
random element $\xi$, see \cite[Th.~5.3]{dav:mol:zuy08}.  In special
cases one obtains the characteristic function, the Laplace transform
or the cumulative distribution function.

In $\R_+^d$ with the $p$-sum operation the characters are given by
\begin{equation}
  \label{eq:chi_-sum_-x_iu}
  \chi_u(x)=\exp\left\{-\sum_{i=1}^d (x_iu_i)^p\right\}\,,\quad x\in\R_+^d\,,
\end{equation}
for $u\in\R^d_+$ if $p$ is finite.  If $p=\infty$, the characters are
\begin{equation}
  \label{eq:chi_ux=-begincases-1}
  \chi_u(x)=
  \begin{cases}
    1 & \text{if}\; x_iu_i\leq 1\; \text{for all}\; i=1,\dots,d,\\
    0 & \text{otherwise}\,,
  \end{cases}
\end{equation}
for $u\in\R_+^d$, so that $\E\chi_u(\xi)=\Prob{\xi\leq u^{-1}}$ with
$u^{-1}=(u_1^{-1},\dots,u_d^{-1})$.

\subsection{$L_1(p)$-zonoids}
\label{sec:l_1p-zonoids}

Let
\begin{equation}
  \label{eq:y-product}
  yM=\{(y_1x_1,\dots,y_dx_d):\; x\in M\}
\end{equation}
denote the set $M\subset\R^d$ \emph{rescaled} by a vector $y\in\R^d$
and a set $M\subset\R^d$.  

\begin{definition}
  \label{def:mgen-z}
  Let $\sigma$ be a finite measure on $\SSp$ with total mass $c$.
  Define $\eta$ to be a random vector distributed according to
  $c^{-1}\sigma$. The set $K=c\E X$ for
  \begin{equation}
    \label{eq:x=et-dots-eta_dv_d}
    X=\eta B_q^d=\{(\eta_1v_1,\dots,\eta_dv_d):\; \|v\|_q\leq 1,\; v\in\R^d\}
  \end{equation}
  with $p^{-1}+q^{-1}=1$ for $p\geq1$ is said to be
  \emph{$L_1(p)$-zonoid} with spectral measure $\sigma$.
\end{definition}

Definition~\ref{def:mgen-z} can be reformulated for a probability
measure $\sigma$ on $\R^d$ and the corresponding random vector $\eta$.
In this case $K=\E(\eta B_q^d)$ is also called the $L_1(p)$-zonoid
generated by $\eta$.

Note that $X$ from (\ref{eq:x=et-dots-eta_dv_d}) is a rescaled
$\ell_q$-ball $B_q^d$.  If $p=\infty$, then $X$ becomes a rescaled
crosspolytope.  More generally, taking the Firey $\alpha$-expectation
$\E_\alpha X$ yields an $L_\alpha(p)$-zonoid.  The conventional and
$L_p$-zonoids are not members of these new families. It is however
possible to define a family of sets that includes all zonoids
introduced so far.

\begin{definition}
  \label{def:all-z}
  Let $M$ be a centred star-shaped set in $\R^d$ and let $\eta$ be a
  random vector in $\R^d$ with $\E\|\eta\|^p<\infty$ for $p\geq1$.
  The Firey $p$-expectation of $\eta M$ is called
  \emph{$L_p(M)$-zonoid}.
\end{definition}

If $M$ is the segment with end-points $\pm(1,\dots,1)$, then
Definition~\ref{def:all-z} yields the family of $L_p$-zonoids.  The
case of $M$ being simplices of varying dimension was considered in
\cite{ric82}.  If $M$ is an $\ell_q$-ball, we arrive at
Definition~\ref{def:mgen-z}.  Although Definition~\ref{def:all-z} with
a general $M$ may be of geometric interest, we do not pursue its study
in this paper.

It is obvious that $L_1(p)$-zonoids are \emph{plane-symmetric}, i.e.
they are symmetric with respect to all coordinate planes.  The
following proposition shows that the $L_1(p)$-zonoid is actually
determined by the vector $|\eta|=(|\eta_1|,\dots,|\eta_d|)$ of the
absolute values of $\eta=(\eta_1,\dots,\eta_d)$. This means that it
suffices to consider only spectral measures on $\R_+^d$.

\begin{prop}
  \label{prop:whole-space}
  If $|\eta'|$ and $|\eta''|$ share the same distribution, then the
  $L_1(p)$-zonoids generated by $\eta'$ and $\eta''$ coincide.
\end{prop}
\begin{proof}
  It suffices to notice that $\eta' B_q^d$ and $\eta'' B_q^d$ coincide
  in distribution.
\end{proof}

\begin{theorem}
  \label{thr:psum-s1s}
  A random vector $\xi\in\R^d_+$ is strictly stable for $p$-sums with
  $\alpha=1$, $p\in(1,\infty]$ and spectral measure $\sigma$ if and
  only if
  \begin{equation}
    \label{eq:e-chi_uxi=e-hk}
    \E \chi_u(\xi)=e^{-h(K,u)}\,,\quad u\in\R_+^d\,,
  \end{equation}
  for an $L_1(p)$-zonoid $K$ with spectral measure
  $\Gamma(1-\frac{1}{p})\sigma$, where $\chi_u$ is given by
  (\ref{eq:chi_-sum_-x_iu}) if $p$ is finite and by
  (\ref{eq:chi_ux=-begincases-1}) if $p=\infty$.
\end{theorem}
\begin{proof}
  Assume that $p\in(1,\infty)$ and consider $\alpha\in[1,p)$.  The
  general results from \cite[Sec.~5.3]{dav:mol:zuy08} imply that the
  Laplace transform of a strictly stable for $p$-sums random vector in
  $\R^d_+$ with characteristic exponent $\alpha$ is given by
  \begin{displaymath}
    \E\chi_u(\xi)=e^{-\psi(u)}\,,
  \end{displaymath}
  where
  \begin{equation}
    \label{eq:phiu=-chi_-quad}
    \psi(u)=\int_{\R_+^d}(1-\chi_u(x))\Lambda(dx)\,, \quad u\in\R_+^d\,,
  \end{equation}
  and $\Lambda$ is the L\'evy measure $\xi$. The L\'evy measure admits
  the polar decomposition as $\alpha t^{-\alpha-d}dt\sigma(dy)$ for
  $x=ty$, so that a change of variables in the integral yields that
  \begin{displaymath}
    \psi(u)=\int_{\SSp} \left(\sum_{i=1}^d
      (u_iy_i)^p\right)^{\alpha/p}\sigma(dy)
    \frac{\alpha}{p}\int_0^\infty (1-e^{-s})s^{-\frac{\alpha}{p}-1}ds\,.
  \end{displaymath}
  Thus,
  \begin{equation}
    \label{eq:phiu=g-frac-int_ssp}
    \psi(u)=\Gamma(1-\frac{\alpha}{p})
    \int_{\SSp} \left(\sum_{i=1}^d(u_iy_i)^p\right)^{\alpha/p}\sigma(dy)\,.
  \end{equation}
  Since the $\ell_p$-norm of $u$ can be written as $(\sum
  u_i^p)^{1/p}=h(B_q^d,u)$, 
  \begin{displaymath}
    \left(\sum_{i=1}^d (y_iu_i)^p\right)^{1/p}=h(yB_q^d,u)\,,
  \end{displaymath}
  where $yB_q^d$ is defined as in (\ref{eq:y-product}). If $\alpha=1$,
  then
  \begin{displaymath}
    \psi(u)=\Gamma(1-\frac{\alpha}{p})\int_{\SSp} h(yB_q^d,u)\sigma(dy)\,, 
  \end{displaymath}
  i.e. $\psi(u)=h(c\E X,u)$ for $u\in\R_+^d$, where
  $c=\sigma(\SSp)\Gamma(1-\frac{1}{p})$ and $X$ given by
  (\ref{eq:x=et-dots-eta_dv_d}) with $p^{-1}+q^{-1}=1$ and $\eta$
  distributed according to the normalised $\sigma$. The case
  $p=\infty$ is considered similarly, see \cite{mol06} for the special
  study of max-stable laws.
\end{proof}

\begin{remark}[Arithmetic sums]
  \label{rem:a-sums}
  Although the conventional case of arithmetic sums ($p=1$) is not
  covered by Theorem~\ref{thr:psum-s1s}, (\ref{eq:e-chi_uxi=e-hk})
  also holds.  Then $q=\infty$, so that (\ref{eq:x=et-dots-eta_dv_d})
  reads $X=\times_{i=1}^d[-\eta_i,\eta_i]$ for $\eta\in\R_+^d$. Thus
  $K=\E X=\times_{i=1}^d [-\E\eta_i,\E\eta_i]$, so that
  \begin{displaymath}
    h(K,u)=\sum u_i\E \eta_i\,, \quad u\in\R_+^d\,,
  \end{displaymath}
  i.e. $\xi$ is deterministic. Indeed, one-sided strictly stable laws
  with $\alpha=1$ are necessarily degenerated.
\end{remark}

\begin{example}[Max-zonoids]
  \label{ex:max-z}
  Let $\xi$ be a max-stable random vector in $\R_+^d$, whose marginals
  are unit Fr\'echet, i.e. $\xi$ is stable with respect to the
  coordinatewise maximum operation and exponent $\alpha=1$. The
  Laplace transform of $\xi$ is the cumulative distribution function
  at the point $u^{-1}$ and
  \begin{displaymath}
    \Prob{\xi\leq u^{-1}}=e^{-h(K,u)}
  \end{displaymath}
  for an $L_1(\infty)$-zonoid $K$, i.e. $K$ is the expectation of the
  randomly rescaled crosspolytope in $\R^d$. These zonoids (more
  exactly their intersections with $\R_+^d$) have been explored in
  \cite{mol06} in view of the studies of max-stable distributions, and
  so are called there \emph{max-zonoids}.
\end{example}

\begin{example}[$p=2$]
  \label{ex:p=2}
  If $\xi$ is strictly stable for $p$-sums with $p=2$ and
  $\alpha=1$, then the Laplace transform of $\xi$ is given by
  \begin{displaymath}
    \E \exp\left\{-\sum (\xi_iu_i)^2\right\}= e^{-h(K,u)}\,.
  \end{displaymath}
  The $L_1(2)$-zonoid $K$ is the selection expectation of a randomly
  rescaled Euclidean ball, i.e. the centred plane-symmetric random
  ellipsoid. 
\end{example}

\begin{example}
  \label{ex:2-dim}
  In dimension $d=2$ it is possible to calculate the support function
  of $X$ from (\ref{eq:x=et-dots-eta_dv_d}) for $u=(u_1,u_2)\in\R_+^2$
  as
  \begin{align*}
    h(X,u)=\sup\{u_1\eta_1\cos^{2/q}\theta+u_2\eta_2\sin^{2/q}\theta:\;
    0\leq\theta\leq \frac{\pi}{2}\}\,.
  \end{align*}
  Substituting the value of the critical point
  $\tan\theta=(u_2\eta_2/u_1\eta_1)^{p/2}$ and noticing that
  $\eta_1,\eta_2\geq0$ we arrive at 
  \begin{displaymath}
    h(K,u)=\E h(X,u)
    =\E \frac{(u_1\eta_1)^p+(u_2\eta_2)^p}{(u_1\eta_1+u_2\eta_2)^{p-1}}\,. 
  \end{displaymath}
\end{example}

\begin{example}[Completely dependent and independent cases]
  \label{ex:sp-c}
  If $\eta$ is deterministic, then the corresponding $L_1(p)$-zonoid
  is a rescaled $\ell_q$-ball and the coordinates of $\xi$ are
  completely dependent. 
  
  A centred parallelepiped $\times_{i=1}^d [-a_i,a_i]$ is an
  $L_1(p)$-zonoid for each $p\geq1$. To check this, it suffices to
  take the spectral measure concentrated at the unit basis vectors
  $e_1,\dots,e_d$ with masses $a_1,\dots,a_d$, so that $\xi$ has
  independent components.  For instance, $X$ from
  (\ref{eq:x=et-dots-eta_dv_d}) becomes the segment $[-e_i,e_i]$ if
  $\eta=e_i$. Therefore, polytopes may be $L_1(p)$-zonoids for $p>1$,
  cf Theorem~\ref{thr:strict-convex}.
\end{example}

Thus, $\eta$ equal to one of the basic vectors results in $X$ from
(\ref{eq:x=et-dots-eta_dv_d}) being a segment, while any $\eta$ from
the interior of $\R_+^d$ results in $X$ being a rescaled
$\ell_q$-ball. By combining such values of $\eta$ it is easy to
construct further examples of $L_1(p)$-zonoids. For instance, if
$\eta$ takes the values $(2,\dots,2)$ and $(2,0,\dots,0)$ with equal
probabilities $1/2$, then the corresponding $L_1(p)$-zonoid is the
Minkowski sum of the unit $\ell_q$-ball and the segment with
end-points $\pm(1,0,\dots,0)$.

\subsection{One-sided strictly stable laws}
\label{sec:one-sided-strictly}

The construction based on $p$-sums makes it possible to provide a
geometric interpretation of strictly stable laws for arithmetic sums
on $\R^d_+$ and $\alpha\in(0,1]$.

\begin{theorem}
  \label{thr:01}
  A random vector $\xi\in\R_+^d$ is strictly stable (for arithmetic
  sums) with $\alpha\in(0,1]$ if and only if the Laplace transform of
  $\xi$ is given by
  \begin{equation}
    \label{eq:e-e-sum}
    \E e^{-\langle\xi,u\rangle} = e^{-h(K,u^\alpha)}\,, \quad u\in\R_+^d\,,
  \end{equation}
  where $u^\alpha=(u_1^\alpha,\dots,u_d^\alpha)$ and $K$ is
  $L_1(\alpha^{-1})$-zonoid called the associated zonoid of $\xi$. The
  spectral measure of $K$ is $\Gamma(1-\alpha)\sigma$, where $\sigma$
  is the spectral measure of $\xi$. 
\end{theorem}
\begin{proof}
  The random vector $\xi^\alpha$ is strictly stable for $p$-sums with
  $p=\frac{1}{\alpha}$, so that (\ref{eq:e-e-sum}) follows from
  Theorem~\ref{thr:psum-s1s}. If $\alpha=1$, the law of $\xi$ is
  degenerated, see Remark~\ref{rem:a-sums}.
\end{proof}

In particular, if $\alpha=\frac{1}{2}$, then $K$ from
(\ref{eq:e-e-sum}) is the expectation of a random ellipsoid as in
Example~\ref{ex:p=2}. By comparing (\ref{eq:e-e-sum}) with
(\ref{eq:l-f-one-s}) we see that $\|u^\alpha\|_{F^*}=\|u\|_F^\alpha$, so
that $K^*=F^\spower{\alpha}$.

\begin{example}[One-sided sub-stable laws]
  \label{ex:one-sided-sub}
  Let $\xi$ be one-sided stable law in $\R_+^d$ with $\alpha\in(0,1)$
  and let $\zeta$ be a non-negative stable random variable with
  characteristic exponent $\beta\in(0,1)$. Then
  $\xi'=\zeta^{1/\alpha}\xi$ has a sub-stable distribution, see
  Example~\ref{ex:substable}. It is easily seen that
  \begin{displaymath}
    \E e^{-\langle \xi',u\rangle} = e^{-h(K,u^\alpha)^\beta}
    =e^{-h(L,u^{\alpha\beta})}\,,
  \end{displaymath}
  where $K$ and $L$ are the associated zonoids of $\xi$ and $\xi'$
  respectively, i.e. $K$ is an $L_1(\alpha^{-1})$-zonoid and $L$ is an
  $L_1((\alpha\beta)^{-1})$-zonoid. Note that
  $\|u\|_F^\beta=\|u^\beta\|_{F^\spower{\beta}}$ for any star body
  $F$, where $F^\spower{\beta}$ is defined by (\ref{eq:set-power}).
  Hence $h(K,u^\alpha)^\beta=h(L,u^{\alpha\beta})$ where
  $L^*=(K^*)^\spower{\beta}$.
\end{example}

\begin{theorem}
  \label{thr:l1p-inclusion}
  If $K$ is an $L_1(p)$-zonoid for $p\geq1$, then $K$ is
  $L_1(r)$-zonoid for all $r>p$.
\end{theorem}
\begin{proof}
  We refer to the construction from Example~\ref{ex:one-sided-sub}.
  Assume that $K$ is a parallelepiped, i.e. the components of $\xi$
  are independent. Then $L^*$ is the $\beta$-power of the
  crosspolytope $K^*$. Since the crosspolytope $K^*$ is the (possibly
  rescaled) $\ell_1$-ball, its $\beta$-power $(K^*)^\spower{\beta}$ is
  the (possibly rescaled) $\ell_{1/\beta}$-ball.  Its polar
  $L=((K^*)^\spower{\beta})^*$ is a (possibly rescaled)
  $\ell_{1/(1-\beta)}$-ball. By the construction of
  Example~\ref{ex:one-sided-sub}, the $\ell_{1/(1-\beta)}$-ball $L$ is
  an $L_1((\alpha\beta)^{-1})$-zonoid for all $\alpha\in(0,1)$. By
  setting $q=1/(1-\beta)$, it is easy to see that $\ell_q$-ball is
  $L_1(r)$-zonoid for all $r>p$, where $p^{-1}+q^{-1}=1$.

  Thus, $\ell_q$-ball can be represented as the expectation of
  rescaled $\ell_{r'}$-balls for each $r'<q$.  Since each
  $L_1(p)$-zonoid is the expectation of the rescaled $\ell_q$-ball, it
  can also be expressed as the expectation of rescaled
  $\ell_{r'}$-ball, where $r'$ is associated with $r$, so that it is
  also an $L_1(r)$-zonoid. 
\end{proof}

Thus, the family of $L_1(p)$-zonoids becomes richer if $p$ increases.
The richest one is the family of $L_1(\infty)$-zonoids (or
max-zonoids). In the planar case it includes all plane-symmetric
convex sets \cite{mol06}, while in the spaces of higher dimensions
this is no longer the case.

The following generalisation of Theorem~\ref{thr:psum-s1s} treats
general strictly stable laws for $p$-sums.

\begin{theorem}
  \label{thr:p-rad}
  A random vector $\xi\in\R_+^d$ is strictly stable for $p$-sums with
  $p\in(0,\infty]$ and the characteristic exponent $\alpha\leq p$ if
  and only if
  \begin{displaymath}
    \E\chi_u(\xi)=\exp\{-h(K,u^\alpha)\}
  \end{displaymath}
  for an $L_1(p/\alpha)$-zonoid $K$.
\end{theorem}
\begin{proof}
  If $\xi$ is strictly stable for $p$-sums with $\alpha\leq p$ and
  finite $p$, then $\xi^p$ is $\saps$ for arithmetic sums with
  $\alpha'=\alpha/p\in(0,1]$, so that Theorem~\ref{thr:01} yields that
  \begin{displaymath}
    \E \exp\left\{-\sum(\xi_iu_i)^p\right\}
    =\exp\{-h(K,u^{\alpha'p})\}=\exp\{-h(K,u^\alpha)\}\,,
  \end{displaymath}
  where $K$ is an $L_1(p/\alpha)$-zonoid. 
  The case of $p=\infty$ follows from the fact that $\psi(u)$ from
  (\ref{eq:phiu=-chi_-quad}) is written as
  \begin{displaymath}
    \psi(u)=\int_{\SSp} (\max_{1\leq i\leq
      d}(u_iy_i))^\alpha\sigma(dy)
    =\int_{\SSp} \max_{1\leq i\leq d}(u_iy_i^\alpha)\sigma'(dy)
  \end{displaymath}
  for another measure $\sigma'$.
\end{proof}

\subsection{Moments of one-sided stable laws}
\label{sec:dens-moments-one}

Similar to the Fourier analysis technique in
Section~\ref{sec:moments-sas-laws}, it is possible to use the
expression for the Laplace transform of one-sided stable law in order
to obtain more information about its probability density function $f$
and moments. By integrating the both parts of
\begin{equation}
  \label{eq:int_r_+d-e-langle}
  \int_{\R_+^d} e^{-\langle x,u\rangle} f(x)dx=e^{-h(K,u^\alpha)}\,,
  \quad u\in\R_+^d\,,
\end{equation}
with respect to $u$ with a certain weight, we arrive at various
expressions for the moments of $\xi$. To start with, integrate the
both sides of (\ref{eq:int_r_+d-e-langle}) over the ray $\{tu:\;
t\geq0\}$ with weight $t^\lambda$ for a fixed unit vector $u$ and
$\lambda>-1$. Since
\begin{displaymath}
  \int_0^\infty e^{-\langle x,u\rangle t}t^\lambda dt
  =\langle x,u\rangle^{-\lambda-1}\Gamma(1+\lambda)\,,
\end{displaymath}
the integration of the right-hand side yields that
\begin{displaymath}
  \E \langle\xi,u\rangle^{-\lambda-1}=\frac{1}{\alpha}
  \frac{\Gamma(\frac{1+\lambda}{\alpha})}{\Gamma(1+\lambda)}
  h(K,u^\alpha)^{-(\lambda+1)/\alpha}\,.
\end{displaymath}
For instance, if $K$ is the parallelepiped $\times_{i=1}^d[-a_i,a_i]$
(which corresponds to the independent coordinates of $\xi$), then for
$\lambda=0$ we have
\begin{displaymath}
  \E\langle\xi,u\rangle^{-1}
  =\Gamma(1+\frac{1}{\alpha})\left(\sum
    a_iu_i^\alpha\right)^{-1/\alpha}
  =\left(\sum \left(\frac{u_i}{\E(\xi_i^{-1})}\right)^\alpha\right)^{-1/\alpha}\,.
\end{displaymath}
Note that we have used the fact that
$\E\xi_i^{-1}=\Gamma(1+1/\alpha)a_i^{-1/\alpha}$. 

If one performs a similar integration with
$\lambda\in(-1-\alpha,-1)$, it is possible to regularise the integral
of $e^{-\langle x,ut\rangle}$ by subtracting the value of the function
for $t=0$ as
\begin{displaymath}
  \int_{\R_+^d} \left(1-e^{-\langle x,u\rangle}\right) f(x)dx
  =1-e^{-h(K,u^\alpha)}\,.
\end{displaymath}
Then, for $\beta\in(0,\alpha)$
\begin{displaymath}
  \E\langle\xi,u\rangle^\beta
  =\frac{\Gamma(1-\frac{\beta}{\alpha})}{\Gamma(1-\beta)} h(K,u^\alpha)^{\beta/\alpha}\,.
\end{displaymath}

Clearly, the above expressions for the moments can be obtained by
calculating the moments of one-sided stable random variable
$\langle\xi,u\rangle$.  Furthermore, similar results can be obtained
for random vectors which are strictly stable for $p$-sums. The case of
$p=\infty$ is considered in \cite{mol06}.

\medskip

The multivariate \emph{Laplace ordering} is introduced in
\cite{sha:sha94} by pointwise ordering of the Laplace transform. Thus,
two one-sided strictly stable random vectors with the same
characteristic exponent are Laplace ordered if and only if the
corresponding associated zonoids are ordered by inclusion.
Applications of this ordering for actuarial quantities have been
considered in \cite{denu01}.

\section{Geometric interpretations of the spectral measure}
\label{sec:surf-area-repr}

\subsection{$p$-surface area measures and spectral bodies}
\label{sec:p-surface-area}

Assume that the $\sas$ distribution is \emph{full-dimensional}, i.e.
the spectral measure $\sigma$ is not concentrated on a great
sub-sphere of $\Sphere$. The \emph{Minkowski existence problem}
\cite[Sec.~7.1]{schn} establishes that for each finite positive
full-dimensional even Borel measure on the unit sphere there exists a
unique centred convex body $Q$ such that $S(Q,\cdot)=\sigma(\cdot)$.
Here $S(Q,\cdot)$ is the \emph{surface area measure} of $Q$ which is
the unique measure on the unit sphere that satisfies
\begin{displaymath}
  dV_1(Q,L)=\lim_{\eps\downarrow0}\frac{|Q+\eps L|-|Q|}{\eps}
  =\int_{\Sphere} h(L,u)S(Q,du)\,,
\end{displaymath}
where $V_1(Q,L)$ is called the \emph{mixed volume} of the convex sets
$Q$ and $L$. We refer to \cite{schn} for a detailed presentation of the
relevant concepts from convex geometry.

The $L_p$ generalisation of the above concepts has been studied in
\cite{lut93,lut96}. The \emph{$p$-mixed volume} $V_p(Q,L)$ of two
convex bodies containing the origin is defined by
\begin{equation}
  \label{eq:fracdp-v_pq-l}
  \frac{d}{p} V_p(Q,L)
  =\lim_{\eps\downarrow0}\frac{|Q\psum[p] \eps^{1/p}L|-|Q|}{\eps}
  =\frac{1}{p}\int_{\Sphere} h(L,u)^pS_p(Q,du)\,,
\end{equation}
where the \emph{$p$-surface area measure} $S_p(Q,\cdot)$ satisfies
\begin{displaymath}
  S_p(Q,du)=h(Q,u)^{1-p}S(Q,du)\,.
\end{displaymath}
The \emph{$p$-Minkowski problem} is solved in \cite[Th.~3.3]{lut93} by
showing that if $\sigma$ is an even positive Borel measure which is
not concentrated on a great sub-sphere of $\Sphere$ and $p>1$, $p\neq
d$, then there exists a unique centred convex body $Q$, such that
$S_p(Q,\cdot)=\sigma(\cdot)$.  By combining this representation with
the classical Minkowski problem for $p=1$, any full-dimensional
spectral measure $\sigma$ corresponding to $\sas$ law $\xi$ with
$\alpha\in[1,2)$ ($\alpha\in[1,2]$ in dimension $d\geq3$) can be
interpreted as the $\alpha$-surface area measure of a centred convex
body $Q$, i.e.  $\sigma(\cdot)=S_\alpha(Q,\cdot)$. We call $Q$ the
\emph{spectral body} of $\xi$. By (\ref{eq:u_k=int_ss-langle-u}), the
Minkowski functional of the associated star body $F$ (or the support
function of the associated zonoid $K$) can be expressed as
\begin{displaymath}
  \|u\|_F^\alpha=h(K,u)^\alpha=\int_{\Sphere} |\langle u,v\rangle|^\alpha
  S_\alpha(Q,dv)\,.
\end{displaymath}
The \emph{$p$th projection body} $\Pi_pQ$ of $Q$ is defined in
\cite{lut:yan:zhan00} by
\begin{displaymath}
  h(\Pi_pQ,u)^p=\frac{1}{d\kappa_dc_{d-2,p}}\int_{\Sphere}
  |\langle x,u\rangle|^p S_p(Q,dx)\,,
\end{displaymath}
where 
\begin{displaymath}
  c_{d,p}=\frac{\kappa_{d+p}}{\kappa_2\kappa_d\kappa_{p-1}}\,.
\end{displaymath}
The set $\Pi_pQ$ (or its dilated version) is sometimes denoted by
$\Gamma^*_{-p}Q$ and is called the polar centroid body, see, e.g.
\cite[(4.2)]{lut:yan:zhan05}. The normalising constant guarantees that
$\Pi_p B=B$ for the unit Euclidean ball.  Thus, the associated zonoid
$K$ satisfies $K=(d\kappa_dc_{d-2,\alpha})^{1/\alpha}\Pi_pQ$. 

The $L_p$-analogue of the \emph{Petty projection inequality} proved in
\cite[Th.~2]{lut:yan:zhan00} establishes that
\begin{displaymath}
  |Q|^{(d-p)/p}\cdot |\Pi^*_pQ|\leq \kappa_d^{d/p}\,.
\end{displaymath}
Using the fact that the polar set $\Pi^*_\alpha Q$ is a dilate of the
associated star body $F$ of $\xi$ and setting $p=\alpha$, we arrive at
the following inequality
\begin{equation}
  \label{eq:q-1+dalphacdot-fleq}
  |Q|^{-1+d/\alpha}\cdot |F|\leq (dc_{d-2,\alpha})^{-d/\alpha}
\end{equation}
valid for $\alpha\in[1,2]$ with the equality attained in the
sub-Gaussian case. Recall that $|F|$ determines the value of the
density of the stable law at the origin and provides a bound for the
moments of $\|\xi\|$. In the Gaussian case $\alpha=2$ for $d\geq3$ and
(\ref{eq:q-1+dalphacdot-fleq}) reads
\begin{displaymath}
  |Q|^{-1+d/2}\cdot |F|\leq \left(1+\frac{d}{2}\right)^{-d/2}\,.
\end{displaymath}

\subsection{Spectral star body}
\label{sec:spectral-star-body}

The following interpretation of the spectral measure is useful for
$\sas$ laws with arbitrary $\alpha\in(0,2]$. Assume that the spectral
measure $\sigma$ of $\xi$ has a positive continuous density on
$\Sphere$ with respect to the $(d-1)$-dimensional surface area
measure, so that
\begin{displaymath}
  \sigma(du)=\frac{1}{d+\alpha}\rho_L(u)^{d+\alpha}du
\end{displaymath}
for a star body $L$, which we call the \emph{spectral star body} of
$\xi$.  By passing to polar coordinates it is easily seen that
\begin{displaymath}
  \|u\|_F^\alpha = \int_{\Sphere} |\langle u,y\rangle|^\alpha\sigma(dy)
  =\int_L |\langle u,y\rangle|^\alpha dy\,.
\end{displaymath}
The integral in the right-hand side is related to the
\emph{$p$-centroid body} $\Gamma_p L$ and its polar $\Gamma^*_p L$
defined (up to a possibly different normalisation) by
\begin{displaymath}
  h(\Gamma_pL,u)^p=\|u\|_{\Gamma_p^*L}^p
  =\frac{1}{c_{d,p}|L|}\int_L|\langle u,y\rangle|^p dy\,,
\end{displaymath}
see \cite{gard95,lut:yan:zhan00} and \cite[(6.1)]{lut:yan:zhan04}.
Note that $\Gamma_p L$ is convex for all $p\geq1$. Thus, the
associated star body of $\sas$ vector $\xi$ is related to its spectral
star body by
\begin{displaymath}
  F=(c_{d,\alpha}|L|)^{-1/\alpha}\;\Gamma_\alpha^*L\,.
\end{displaymath}
It is proved in \cite{lut:yan:zhan00} that $|\Gamma_p L|\geq|L|$ if
$p\geq1$, which implies the Blaschke-Santal\'o inequality
$|L|\cdot|\Gamma_p^*L|\leq \kappa_d^2$ with equality if and only if
$L$ is a centred ellipsoid, see also \cite{lut:zhan97}.  If
$p=\alpha\geq1$, then
\begin{displaymath}
  |F|\cdot|L|^{1+d/\alpha}\leq \frac{\kappa_d^2}{c_{d,\alpha}}\,.
\end{displaymath}

The same spectral star body $L$ can be used to construct $\sas$ random
vectors $\xi_{\alpha,L}$ with varying characteristic exponent
$\alpha$. If $L$ is convex, then \cite[Prop.~2.1.1]{gian03} yields
that there exists a universal constant $c>0$ such that for all
$u\in\R^d$ and $p>1$
\begin{displaymath}
  \left(\int_{L_1}|\langle u,y\rangle|^p dy\right)^{1/p}
  \leq cp \int_{L_1} |\langle u,y\rangle|dy\,,
\end{displaymath}
where $L_1=|L|^{-1/d}L$ has volume 1. Thus
\begin{displaymath}
  \|u\|_{F_{\alpha,L}}\leq c\alpha|L|^{1-1/\alpha}\|u\|_{F_{1,L}}\,,
\end{displaymath}
where $F_{\alpha,L}$ is the associated star body of $\xi_{\alpha,L}$.
By Theorem~\ref{thr:norm}, this yields an inequality between the
$\lambda$-moments of the norm with $\lambda\in(0,1)$ for an $\sas$
random vector with $\alpha\geq1$ and the Cauchy random vector with the
same spectral star body. For the Cauchy distribution we have
$\alpha=1$, where a number of further inequalities for the volumes of
projection and polar bodies are available, see \cite{gard95,schn}.

If $F$ is the associated star body of an $\sas$ law with convex
spectral body $L$, then \cite[Lemma~3.1.1]{gian03} implies that
\begin{displaymath}
  \|u\|_F^\alpha \geq |L|^{1+\lambda/d}\;
  \frac{\Gamma(\alpha+1)\Gamma(d)}{2e\Gamma(\alpha+d+1)}
  h(L,u)^\alpha\,,
  \quad u\in\Sphere\,.
\end{displaymath}
For instance, this inequality may be used in order to obtain lower
bounds for the moments of $\|\xi\|$. It also means that the associated
zonoid of $\xi$ contains a dilate of $L$.

It is likely that other geometric properties, e.g. curvature, surface
area, other intrinsic volumes, of the spectral sets and associated
star bodies have a bearing in view of the studies of $\sas$ laws.

\subsection{Spectral sets for one-sided laws}
\label{sec:spectral-sets-one}

Let $\xi$ be a one-sided strictly stable random vector. Although its
spectral measure $\sigma$ is supported by $\SSp$, consider its
extension on the whole $\Sphere$ in a plane-symmetric way. Note that
this extension is full-dimensional, so that the Minkowski existence
problem guarantees that $\sigma(\cdot)=S(Q,\cdot)$ for a convex
plane-symmetric body $Q$, also called the spectral body of $\xi$.

Let $M$ be a centred convex set in $\R^d$. If $\eta$ is a random
vector distributed according to the normalised $S(Q,\cdot)$ having the
total mass $c$, then the expectation of $c\eta M$ is given from
\begin{align*}
  h(c\E \eta M,u)&=\int_{\Sphere} h(yM,u) S(Q,dy)\\
  &=\int_{\Sphere} h(uM,y) S(Q,dy)=d V_1(Q,uM)\,.
\end{align*}
If $\xi$ is one-sided strictly stable random vector with
characteristic exponent $\alpha$ and spectral measure $\sigma$, then
choose $M$ to be the $\ell_{1/(1-\alpha)}$-ball, so that
\begin{displaymath}
  \E e^{-\langle\xi,u\rangle}=\exp\{-\Gamma(1-\alpha)d V_1(Q,uB_{1/(1-\alpha)})\}\,.
\end{displaymath}
The first Minkowski inequality $V_1(K,L)\geq |K|^{(d-1)/d}|L|^{1/d}$
(see \cite[Th.~6.2.1]{schn}) together with the formula for the volume
of the unit $\ell_p$-ball (see \cite[p.~11]{pis89}) imply that
\begin{displaymath}
  \E e^{-\langle\xi,u\rangle}\leq
  \exp\left\{-\frac{\Gamma(1-\alpha)^2}{\Gamma(d(1-\alpha))^{1/d}} 
    |Q|^{(d-1)/d} u_1\cdots u_d\right\}\,.
\end{displaymath}

\section{Covariation and regression}
\label{sec:covariation}

\subsection{Bivariate case}
\label{sec:two-dimensional-sas}

The covariation replaces the concept of covariance for $\sas$ vectors.
If $\xi=(\xi_1,\xi_2)$ is $\sas$ in $\R^2$ with $\alpha>1$ and the
spectral measure $\sigma$, then the \emph{covariation} of $\xi_1$ on
$\xi_2$ is defined by
\begin{displaymath}
  [\xi_1,\xi_2]_\alpha =\int_{\Sphere[1]}
  s_1s_2^\spower{\alpha-1}\sigma(ds)\,,
\end{displaymath}
see \cite[Sec.~2.7]{sam:taq94}. It is mentioned in
\cite[Sec.~2.7]{sam:taq94} that the covariation can be equivalently
defined as
\begin{equation}
  \label{eq:xi_1-xi_2_-=frac1}
  [\xi_1,\xi_2]_\alpha =\frac{1}{\alpha} \;
  \frac{\partial \sigma^\alpha(t_1,t_2)}{\partial t_1}\bigg|_{t_1=0,t_2=1}\,,
\end{equation}
where $\sigma(t_1,t_2)$ is the scale parameter of
$Y=t_1\xi_1+t_2\xi_2$, i.e.
\begin{equation}
  \label{eq:sigm-t_2=-}
  \sigma^\alpha(t_1,t_2)=\int_{\Sphere[1]}|t_1s_1+t_2s_2|^\alpha\sigma(ds)\,.
\end{equation}

\begin{theorem}
  \label{thr:covariation}
  If $\xi=(\xi_1,\xi_2)$ is $\sas$ with $\alpha\in(1,2]$ and the
  associated zonoid $K$, then
  \begin{equation}
    \label{eq:cov-nf}
    [\xi_1,\xi_2]_\alpha=x_1x_2^{\alpha-1}\,,
  \end{equation}
  where $T(K,(0,1))=\{(x_1,x_2)\}$ is the support point of $K$ in
  direction $(0,1)$, see (\ref{eq:support-set}). 
\end{theorem}
\begin{proof}
  By Theorem~\ref{thr:strict-convex}, $L_p$-zonoids with $p>1$ are
  strictly convex, so that the support set $T(K,u)$ is indeed a
  singleton for each direction $u$.  The right-hand side of
  (\ref{eq:sigm-t_2=-}) can be identified as $h(K,u)^\alpha$ for $K$
  being the associated zonoid of $\xi$.  The partial derivative in the
  right-hand side of (\ref{eq:xi_1-xi_2_-=frac1}) then becomes the
  directional derivative of $h(K,u)$ in direction $(1,0)$. By
  \cite[Th.~1.7.2]{schn} this derivative can be expressed as
  $h(T(K,(0,1)),(1,0))$. Hence
  \begin{displaymath}
    [\xi_1,\xi_2]_\alpha =h(K,(0,1))^{\alpha-1}h(T(K,(0,1)),(1,0))
    =x_1x_2^{\alpha-1}\,.
  \end{displaymath}
\end{proof}

It is shown in \cite[Lemma~2.7.16]{sam:taq94} that, for all
$p\in(1,\alpha)$,
\begin{equation}
  \label{eq:frace-xi_1x-1exi_2p}
  \frac{\E (\xi_1\xi_2^\spower{p-1})}{\E|\xi_2|^p}
  =\frac{[\xi_1,\xi_2]_\alpha}{[\xi_2,\xi_2]_\alpha}\,.
\end{equation}
Using Theorems~\ref{thr:joint-moments} and~\ref{thr:one-s-p} it is
possible to calculate the moments in the left-hand side explicitly as
\begin{align*}
  \E (\xi_1\xi_2^\spower{p-1})&
  =\frac{\alpha2^{p-1}}{\sqrt{\pi}}\Gamma(2-\frac{p}{\alpha})\;
  \frac{\Gamma(1+\frac{p-1}{\alpha})}{\Gamma(\thf-\frac{p-1}{2})}\;
  \frac{2x_1x_2^{p-1}}{\alpha-p}\,,\\
  \E|\xi_2|^p &=2^px_2^p\;\frac{\Gamma(\frac{1+p}{2})}{\sqrt{\pi}}\;
  \frac{\Gamma(1-\frac{p}{\alpha})}{\Gamma(1-\frac{p}{2})}\,,
\end{align*}
where $x_1$ and $x_2$ are the coordinates of $T(K,(0,1))$. By dividing
these expressions we arrive at $x_1/x_2$, which is exactly the
right-hand side of (\ref{eq:frace-xi_1x-1exi_2p}).

\subsection{Multivariate case}
\label{sec:multivariate-case}

The following result provides covariations for random variables that
belong to a linear span of an $\sas$ random vector.

\begin{theorem}
  \label{thr:covar-subvect}
  Let $\xi$ be $\sas$ in $\R^d$ with $\alpha\in(1,2]$.  If $u',u''$
  are non-zero vectors in $\R^d$, then
  \begin{equation}
    \label{eq:langle-xi-urangle}
    [\langle \xi,u'\rangle,\langle\xi,u''\rangle]_\alpha
    =h(K,u'')^{\alpha-1} h(T(K,u''),u')\,.
  \end{equation}
\end{theorem}
\begin{proof}
  The scale parameter of
  $t_1\langle\xi,u'\rangle+t_2\langle\xi,u''\rangle$ is
  $h(K,t_1u'+t_2u'')$. By differentiating its power with respect to
  $t_1$ as in (\ref{eq:xi_1-xi_2_-=frac1}), we arrive at
  (\ref{eq:langle-xi-urangle}).
\end{proof}

Theorem~\ref{thr:covar-subvect} provides an alternative reformulation
of \cite[Lemma~2.7.5]{sam:taq94}. The right-hand side of
(\ref{eq:langle-xi-urangle}) considered a function of $u'$ is the
support function of the singleton $h(K,u'')^{\alpha-1}T(K,u'')$, and
so is additive with respect to $u'$. In particular, if $u'=(1,1,0)$
and $u''=(0,0,1)$ in $\R^3$, it yields the additivity of the
covariation of $\sas$ random variables with respect to its first
argument.  Similarly, one deduces the additivity of the covariation
with respect to the sum of independent second arguments. Furthermore,
the covariations in the left-hand side of (\ref{eq:langle-xi-urangle})
for all $u',u''$ determine uniquely the associated zonoid $K$.

\begin{example}[$\ell_p$-balls]
  \label{ex:lp-ball-cov}
  Assume that $K$ is the unit $\ell_\alpha$-ball, which corresponds to
  $\sas$ vector $\xi$ with i.i.d. components. The support point
  $T(K,u)$ equals the gradient of $\|u\|_\alpha=h(K,u)$, see
  \cite[Cor.~1.7.3]{schn}. Therefore
  $T(K,u)=\{\|u\|_\alpha^{1-\alpha} u^\spower{\alpha-1}\}$. By
  (\ref{eq:langle-xi-urangle}),
  \begin{displaymath}
    [\langle \xi,u'\rangle,\langle\xi,u''\rangle]_\alpha
    =\langle u',(u'')^\spower{\alpha-1}\rangle\,.
  \end{displaymath}
\end{example}

\subsection{Regression coefficients and linearity conditions}
\label{sec:regr-coeff}

The covariation is used to build regression models for $\sas$
distributions.  By \cite[Th.~4.1.2]{sam:taq94},
\begin{displaymath}
  \E(\xi_1|\xi_2)=\frac{[\xi_1,\xi_2]_\alpha}
  {[\xi_2,\xi_2]_\alpha}\xi_2
  \quad \text{a.s.}
\end{displaymath}
Since $[\xi_2,\xi_2]_\alpha=x_2^\alpha$ for
$T(K,(0,1))=\{(x_1,x_2)\}$, we obtain
\begin{displaymath}
  \E(\xi_1|\xi_2)=\frac{x_1}{x_2} \xi_2 \quad\text{a.s.}
\end{displaymath}
Thus, the regression line is the line passing through the origin and
the support point $T(K,(0,1))$. Therefore, the regression lines are
identical for any two $\sas$ laws that share the same associated
zonoid, e.g. for the Gaussian law and its sub-Gaussian variant.

It is known \cite[Sec.~4.1]{sam:taq94} that multiple regression is not
always linear for $\alpha\in(1,2)$. The necessary and sufficient
conditions for the linearity given in \cite{mil78} can be reformulated
geometrically as follows. 

Consider a convex set $K$ and the one-dimensional subspace $H_x$
spanned by $x\in\R^d$. The \emph{shadow boundary} of $K$ in direction
$x$ is the set $\partial(K+H_x)\cap\partial K$, where $\partial$
denotes the topological boundary in $\R^d$, see
\cite[Def.~3.4.7]{thom96}.

\begin{theorem}
  \label{thr:lin-reg}
  Let $(\xi_1,\dots,\xi_d)$ be an $\sas$ random vector with
  $\alpha\in(1,2]$ and the associated zonoid $K$. Then
  $\E(\xi_1|\xi_2,\dots,\xi_d)$ is linear in $\xi_2,\dots,\xi_d$ if
  and only if the shadow boundary of $K$ in direction
  $e_1=(1,0,\dots,0)$ is a subset of a $(d-1)$-dimensional hyperplane,
  which does not contain $e_1$.
\end{theorem}
\begin{proof}
  By Theorem~3.1 from \cite{mil78}, the conditional expectation is
  linear if and only if, for all $u_2,\dots,u_d$,
  \begin{displaymath}
    \frac{\partial}{\partial u_1}\phi_\xi(u_1,u_2,\dots,u_d)\Big|_{u_1=0}
    =\sum_{i=2}^d a_i
    \frac{\partial}{\partial u_i}\phi_\xi(0,u_2,\dots,u_d)\,.
  \end{displaymath}
  By differentiating (\ref{eq:z-rep}) and using \cite[Th.~1.7.2]{schn}
  for the directional derivative of the support function, it is easily
  see that this holds if and only if
  \begin{equation}
    \label{eq:htk-u_1-e_1=sum_i=2d}
    h(T(K,u|_1),e_1)=\sum_{i=2}^d a_i h(T(K,u|_1),e_i)\,,
  \end{equation}
  where $u|_1$ is $u$ with the first coordinate replaced by zero.
  Since $h(T(K,u|_1),e_i)$ is the $i$th coordinate of the singleton
  $T(K,u|_1)$, (\ref{eq:htk-u_1-e_1=sum_i=2d}) means that $T(K,u|_1)$
  is orthogonal to $a=(1,-a_2,\dots,-a_d)$ for all
  $u|_1=(0,u_2,\dots,u_d)$. In other words, the shadow boundary of $K$
  in direction $e_1$ lies in the hyperplane orthogonal to $a$. By
  the condition, the first coordinate of $a$ is not zero, so that this
  hyperplane does not contain $e_1$.
\end{proof}

\begin{example}[Sub-Gaussian laws]
  If $\xi$ is sub-Gaussian with the norm $\|u\|_E=\langle
  Cu,u\rangle$, then $K=E^*$ is an ellipsoid. It is easy to see (see
  \cite{juh95}) that $T(K,u)$ is the point $C^{-1}u/\sqrt{\langle
    C^{-1}u,u\rangle}$. Then the condition of
  Theorem~\ref{thr:lin-reg} holds with $a=Ce_1$.
\end{example}

\begin{cor}
  \label{cor:mult-lin}
  Let $(\xi_1,\dots,\xi_d)$ be an $\sas$ random vector with
  $\alpha\in(1,2]$, the associated zonoid $K$ and the spectral measure
  $\sigma$. Then $\E(\xi_1|\xi_2,\dots,\xi_d)$ is linear in
  $\xi_2,\dots,\xi_d$ if and only if there exists $a\in\R^d$ with
  non-vanishing first coordinate such that one of the following
  equivalent conditions holds for all $u$ orthogonal to $e_1$:
  \begin{align}
    \label{eq:grad-0}
    \langle \grad h(K,u),a\rangle&=0\,,\\
    \label{eq:sigma-grad}
    \int_{\Sphere} \langle y,a\rangle \langle
    y,u\rangle^\spower{\alpha-1}\sigma(dy)&=0\,,\\
    \label{eq:covar-grad}
    [\langle\xi,a\rangle,\langle\xi,u\rangle]_\alpha&=0\,.
  \end{align}
\end{cor}
\begin{proof}
  Since the support function of $K$ is differentiable by
  Theorem~\ref{thr:strict-convex}, the support point of $K$ in
  direction $u$ is given by the gradient of $h(K,u)$, see
  \cite[Cor.~1.7.3]{schn}. This yields, (\ref{eq:grad-0}). By
  differentiating (\ref{eq:hz}), it is easy to see that
  \begin{displaymath}
    \grad h(K,u)=h(K,u)^{\alpha-1}\int_{\Sphere} y
    \langle u,u\rangle^\spower{\alpha-1}\sigma(dy)\,,
  \end{displaymath}
  so that (\ref{eq:grad-0}) is indeed equivalent to
  (\ref{eq:sigma-grad}). Finally, (\ref{eq:langle-xi-urangle}) implies
  that 
  \begin{displaymath}
    [\langle\xi,a\rangle,\langle\xi,u\rangle]_\alpha
    =\int_{\Sphere} \langle y,a\rangle \langle
    y,u\rangle^\spower{\alpha-1}\sigma(dy)\,.
  \end{displaymath}
\end{proof}

The vector $\xi$ has the \emph{multiple regression property} if, for
each linear transformation $A$, the multiple regression of the first
coordinate of $\eta=A\xi$ onto the remaining coordinates is linear.
Note that $\eta$ has the associated zonoid $AK$. By
Theorem~\ref{thr:lin-reg}, this happens if and only if the shadow
boundary of $K$ in each direction is contained in a
$(d-1)$-dimensional hyperplane. W.~Blaschke proved in 1916 that for
dimension $d\geq3$ this is the case if and only if $K$ is an
ellipsoid, see \cite[Th.~3.4.8]{thom96}. By Theorem~\ref{thr:lin-reg},
this geometric result translates into the multiple regression
criterion from \cite[Prop.~4.1.7]{sam:taq94}.

\section{Operations with associated sets}
\label{sec:oper-with-assoc}


If $\xi'$ and $\xi''$ are independent $\sas$ with associated star
bodies $F_1$ and $F_2$, then $\xi=\xi'+\xi''$ has the characteristic
function
\begin{displaymath}
  \E e^{i\langle u,\xi\rangle}
  =\exp\{-(\|u\|_{F_1}^\alpha+\|u\|_{F_2}^\alpha)\}\,.
\end{displaymath}
Thus, $\xi$ has the associated star body $F$ being the radial sum
(also called $\alpha$-star sum) of $F_1$ and $F_2$, i.e.
$F=F_1\spsum[\alpha] F_2$.

\begin{theorem}
  \label{thr:r-sum}
  If $\xi'$ and $\xi''$ are independent $\sas$ in $\R^d$ with
  $\alpha\in[1,2]$ and probability densities $f_{\xi'}$ and
  $f_{\xi''}$ respectively, then the probability density of
  $\xi=\xi'+\xi''$ satisfies
  \begin{displaymath}
    f_\xi(0)^{-\alpha/d}\geq f_{\xi'}(0)^{-\alpha/d}+f_{\xi''}(0)^{-\alpha/d}
  \end{displaymath}
  with the equality if and only if the associated star bodies of
  $\xi'$ and $\xi''$ are dilates. 
\end{theorem}
\begin{proof}
  The result follows from (\ref{eq:f0=fr-gamm-}) and the dual
  Brunn--Minkowski inequality for radial sums of star bodies, see
  \cite[Prop.~1.12]{lut96}.
\end{proof}

The following result concerns approximation by sub-Gaussian laws. 

\begin{theorem}
  \label{thr:ss-approx}
  A law is $\sas$ with $\alpha\in[1,2]$ if and only if it can be
  obtained as a weak limit for the sums of independent sub-Gaussian
  laws with the same characteristic exponent.
\end{theorem}
\begin{proof}
  By \cite[Cor.~6.14]{grin:zhan99}, each centred convex body $F$ is an
  $L_p$-ball with $p\geq1$ if and only if $\|u\|_F^p$ can be uniformly
  approximated for $u$ from the unit sphere by finite sums of the form
  $\|u\|_{E_1}^p+\cdots+\|u\|_{E_m}^p$, where $E_1,\dots,E_m$ are
  centred ellipsoids. The proof is completed by setting $p=\alpha$ and
  using the fact that $\exp\{-\|u\|_{E_i}^\alpha\}$ is the
  characteristic function of a sub-Gaussian law.
\end{proof}

\medskip


Consider maximisation of $\E|\langle\xi,u\rangle|^\lambda$ over
$u\in\R^d$ for fixed $\lambda\in(0,\alpha)$ under the constraints
$\langle u,\mu\rangle=r$ for some $\mu\in\R_+^d$, $r\geq0$, and
$\langle u,(1,\dots,1)\rangle=1$. By Theorem~\ref{thr:sc-prod}, this
is equivalent to maximising $\|u\|_F$ for $u$ satisfying the
constraints, i.e. its solution is the direction of the smallest
radius-vector function for the set $F\cap H$, where
\begin{displaymath}
  H=\{u\in\R^d:\;\langle u,\mu\rangle=r,\; \langle
  u,(1,\dots,1)\rangle=1\}\,.
\end{displaymath}
This corresponds to the idea of portfolio selection studied in
\cite{bel:veh:wal00} for $\alpha>1$.

It is also possible to consider further optimisation problems for the
moments of the norm of $\eta=A\xi$, where $A$ is an invertible linear
transform and $\alpha\in[1,2]$. The direct computation shows that
$\eta$ has the associated zonoid $AK$ and the associated star body
$(A^\top)^{-1}F$. By Theorem~\ref{thr:norm} minimising
$\E\|A\xi\|^\lambda$ for $\lambda\in(-d,\alpha)$ over
$A\in\mathrm{SL}_n$, corresponds to the minimisation of the integral
of $\|A^\top u\|_F$ over $u\in\Sphere$.

Consider a special case of this problem for $\lambda=1$ and
$\alpha\in(1,2]$. In terms of the associated zonoid $K=F^*$, we can
equivalently minimise the mean width
\begin{displaymath}
  w(AK)=2\int_{\Sphere} h(AK,u) du\,,\quad 
\end{displaymath}
over $A\in\mathrm{SL}_n$. It is shown in \cite{gian:mil:rud00} that
$AK$ has the minimal width position if the measure on the unit sphere
with density $h(AK,\cdot)$ is isotropic.

\medskip


Taking a subvector of $\xi$ corresponds to a section of the
associated star body $F$ by the corresponding coordinate subspace. By
applying orthogonal transformations, we see that the \emph{projection}
of $\xi$ on any subspace $H$ has the associated star body $(F\cap
H)\oplus H^\perp$, i.e. the direct sum of $F\cap H$ and the space
orthogonal to $H$. Therefore, the values at the origin of the
probability density function of the projected $\xi$ are closely
related to the intersection body of $F$.

A bound on the volume of a convex set using volumes of its
$(d-1)$-dimensional sections (see \cite{mey88} and
\cite[p.~341]{gard95}) yields the following inequality for the values
of the density function at the origin
\begin{displaymath}
  f(0)^{d-1}\geq \frac{\Gamma(1+\frac{d}{\alpha})^{d-1}}
  {\Gamma(1+\frac{d-1}{\alpha})^d}\;
  \frac{(d-1)!^d}{(d!)^{d-1}}\; \prod_{i=1}^d f_{-i}(0)\,,
\end{displaymath}
where $f_{-i}(0)$ is the density at the origin of the subvector of
$\xi$ with the $i$th coordinate excluded. This inequality holds for
all $\sas$ laws with convex associated star bodies. In the bivariate
case, 
\begin{displaymath}
  f(0)\geq
  \frac{\Gamma(1+\frac{2}{\alpha})}{2\Gamma(1+\frac{1}{\alpha})^2}
  f_1(0)f_2(0)\,,
\end{displaymath}
where $f_1$ and $f_2$ are the marginal densities. Note that the
coordinates of $\xi=(\xi_1,\xi_2)$ can be independent with a convex
$F$ only if $\alpha\in[1,2]$. 

\begin{theorem}
  \label{thr:indep-part}
  Consider two $\sas$ random vectors $\xi'$ and $\xi''$ of dimensions
  $d_1$ and $d_2$ and the random vector $\eta=(\xi',\xi'')$ obtained by
  concatenating $\xi'$ and $\xi''$. Decompose $\R^d$ into the direct
  sum of two linear subspaces $H_1$ and $H_2$ of dimensions $d_1$ and
  $d_2$. 

  Then $\xi'$ and $\xi''$ are independent if and only if the
  associated star body $F$ of $\eta$ (or associated zonoid $K$ if
  $\alpha\in[1,2]$) is the $\alpha$-star sum of $F_1=(F\cap H_1)\times
  H_2$ and $F_2=H_1\times (F\cap H_2)$ (respectively $K$ is the Firey
  $\alpha$-sum of the projection of $K$ onto $H_1$ and onto $H_2$ if
  $\alpha\in[1,2]$).
\end{theorem}
\begin{proof}
  Note that $\xi'$ and $\xi''$ are independent if and only if
  \begin{displaymath}
    \|(u_1,u_2)\|_F^\alpha=\|u_1\|_{F_1}^\alpha+\|u_2\|_{F_2}^\alpha\,,
    \quad u_1\in H_1,\; u_2\in H_2\,.
  \end{displaymath}
  It remains to observe that the associated star bodies of subvectors
  appear as the intersections of $F$ with $H_1$ and $H_2$ and the
  associated zonoids are projections of $K$ onto $H_1$ and $H_2$
  respectively.
\end{proof}


The duality operation transforms convex bodies into their polar sets.
This operation does not generally preserve the property of a set being
a zonoid or $L_p$-ball. However, it makes sense if applied to the
spectral sets of $\sas$ laws. If $\xi$ is $\sas$ with spectral set
$Q$, then its \emph{spectral dual} $\xi^*$ has the spectral measure
$\sigma^*(\cdot)=S_\alpha(Q^*,\cdot)$. Probabilistic studies of this
operation call for geometric results concerning $p$th projection and
centroid bodies of polar sets.

\medskip


Now explore the ordering of $\sas$ vectors based on inclusion
relationship for their associated star bodies. Write $\eta\preceq\xi$
if $F_\xi\subset F_\eta$ for their associated star bodies $F_\xi$ and
$F_\eta$.

\begin{theorem}
  \label{thr:inclusion}
  If $F_\xi\subset F_\eta$ for the associated star bodies of $\sas$
  random vectors $\xi$ and $\eta$ with $\alpha\in(0,2]$, then there
  exist $\tilde\xi\deq\xi$ and $\tilde\eta\deq\eta$ such that
  $|\langle\tilde\xi,u\rangle|\geq |\langle\tilde\eta,u\rangle|$ a.s.
  simultaneously for all $u\in\Sphere$.
\end{theorem}
\begin{proof}
  Fix $u\in\Sphere$. Since $\langle \xi,u\rangle$ and $\langle
  \eta,u\rangle$ are $\sas$ random variables with scale parameters
  $\|u\|_{F_\xi}\geq \|u\|_{F_\eta}$, it is possible to define $\xi$
  and $\eta$ on the same probability space, so that
  $|\langle\tilde\xi,u\rangle|\geq|\langle\tilde\eta,u\rangle|$ a.s.
  By repeating the same argument, it is possible to show that finite
  dimensional distributions of the $|\langle\xi,u\rangle|$,
  $u\in\Sphere$, are stochastically greater than the
  finite dimensional distributions of $|\langle\eta,u\rangle|$,
  $u\in\Sphere$. The statement follows from the continuity of the
  processes, see also \cite[Th.~4]{kam:kre:obr77}.
\end{proof}

\begin{definition}
  \label{def:b-m-dist}
  If $F_1$ and $F_2$ are two convex sets representing the unit balls
  in $\R^d$ with two norms, then the \emph{Banach--Mazur distance}
  $\rho_{BM}(F_1,F_2)$ between $F_1$ and $F_2$ (or between the
  corresponding normed spaces) is defined as the infimum of $t>0$ such
  that $F_1\subset AF_2\subset tF_1$ for an invertible matrix $A$, see
  \cite[Sec.~2.1]{lin:mil93}.
\end{definition}

The \emph{Banach--Mazur distance} $\rho_{BM}(\xi,\eta)$ between two
$\sas$ vectors $\xi$ and $\eta$ with $\alpha\in[1,2]$ (needed to
ensure the convexity of the associated star bodies) is defined as the
Banach--Mazur distance between their associated star bodies. By
Theorem~\ref{thr:inclusion}, $\rho_{BM}(\xi,\eta)$ is the infimum of
$t>0$ such that $\xi\preceq A\eta\preceq t\xi$ for an invertible
matrix $A$. It is well known that the Banach--Mazur distance between
any $d$-dimensional space and $\R^d$ with the elliptical norm is at
most $\sqrt{d}$. Therefore, for each $\sas$ random vector $\xi$ there
exists a sub-Gaussian random vector $\eta$ with the same
characteristic exponent such that $\rho_{BM}(\xi,\eta)\leq\sqrt{d}$
and both $\eta$ and $\xi$ can be realised on the same probability
space as $\tilde\xi$ and $\tilde\eta$ so that
\begin{displaymath}
  |\langle\tilde\eta,u\rangle|\leq |\langle A\tilde\xi,u\rangle|
  \leq \sqrt{d}|\langle \tilde\eta,u\rangle|\quad \text{a.s.}
\end{displaymath}
holds simultaneously for all $u$. It is known that
$(\R^d,\|\cdot\|_p)$ is the farthest from the Euclidean among all
subspaces of $L_p([0,1])$, see \cite[Sec.~5.1]{lin:mil93}. Thus, the
$\sas$ law with independent components is the farthest one from the
sub-Gaussian law with the same characteristic exponent.

\emph{Dvoretzky's theorem} states that if a natural number $n$ and
$\eps>0$ are given, then every normed space of sufficiently large
dimension $d$ (depending on $n$ and $\eps$) has an $n$-dimensional
subspace, whose Banach--Mazur distance from $\R^d$ with an elliptical
norm is less than $\eps$.  Since section of star bodies correspond to
projections of $\sas$ vectors, Dvoretzky's theorem implies that each
$\sas$ vector with convex associated star body and of sufficiently
high dimension can be projected onto an $n$-dimensional subspace, such
that its projection lies arbitrarily close to a sub-Gaussian law.

\medskip


From representation (\ref{eq:ch-f-rep}) for the characteristic
function one immediately obtains that if $F_1,F_2,\ldots$ is a
sequence of star bodies corresponding to $\sas$ vectors
$\xi_1,\xi_2,\ldots$ with fixed $\alpha\in (0,2]$, then
$\xi_n\dto\xi$ (converge in distribution) if and only if $\xi$ is an
$\sas$ law with associated star body $F$ satisfying
$\|u\|_{F_n}\to\|u\|_F$ as $n\ti$ for all $u\in\R^d$. It is also
possible to provide a version of this result for distributions from
the domain of attraction of $\sas$ laws.

\begin{definition}
  \label{def:level-t}
  If $\eta$ is a random vector in $\R^d$, then its associated star
  body at level $t$ is the star-shaped set $F_t$ obtained by
  (\ref{eq:u_k=int_ss-langle-u}) using the spectral measure $\sigma_t$
  given by
  \begin{displaymath}
    \sigma_t(A)=\Prob{\frac{\eta}{\|\eta\|}\in A\mid \|\eta\|\geq
      t}\,,\quad t>0\,.
  \end{displaymath}
\end{definition}

The classical limit theorem for convergence to stable random vectors
with $\alpha\in(0,2)$ implies that if $\eta$ belongs to the domain of
attraction of $\sas$ law $\xi$ if and only if $\|\eta\|$ has a
regularly varying tail and $\sigma_t$ converges weakly to $\sigma$
being the spectral measure of $\xi$, see \cite{ara:gin}. 

If $\alpha\geq1$, the weak convergence of measures $\sigma_t$ is
equivalent to the Hausdorff convergence of the corresponding
$L_\alpha$-zonoids $K_t=F_t^*$. This can be proved in the same way as
for $\alpha=1$ on \cite[p.~184]{schn}. For general $\alpha\in(0,2]$, the
weak convergence of $\sigma_t$ to $\sigma$ is equivalent to the
pointwise convergence of norms $\|u\|_{F_t}$ for $u\in\Sphere$
together with the convergence of the integrals of the norms over the
unit sphere. Note that the latter convergence implies the convergence
of the total masses of $\sigma_t$.

\section{James orthogonality}
\label{sec:james-orthogonality}

The associated zonoid $K_\xi$ can be used as the \emph{scale
  parameter} of $\sas$ random vector $\xi$ in case $\alpha\in[1,2]$.
For general $\alpha\in(0,2]$, the star body plays the role of the
inverse scale parameter. Based on this observation, it is possible to
generalise several concepts that have been defined only in the
univariate and bivariate cases or for $\alpha>1$.

The \emph{covariation norm} $\pnn{\eta}$ of $\sas$ random variable
$\eta$ is defined to be the scale parameter of $\eta$, i.e.
$\pnn{\eta}=a$ if and only if $\phi_\eta(u)=e^{-a^\alpha u^\alpha}$,
see \cite[Sec.~2.9]{sam:taq94}.  The family $\sS_\xi$ of $\sas$ random
variables obtained as linear combinations of the coordinates of $\sas$
random vector $\xi$ in $\R^d$ becomes a normed space if $\sS_\xi$ is
equipped with the covariation norm. If $\xi$ has the associated star
body $F$ and $\eta=\langle u,\xi\rangle$, then $\pnn{\eta}=\|u\|_F$,
i.e.  $(\sS_\xi,\pnn{\cdot})$ is isometric to $(\R^d,\|\cdot\|_F)$.

The definition of normality in normed linear spaces goes back to
G.~Birkhoff (1935), see \cite[Sec.~3.2]{thom96}. If $(X,\|\cdot\|)$ is
a normed linear space, then $x$ is normal to $y$ (notation $x\dashv
y$) if $\|x+cy\|\geq\|x\|$ for all $c\in\R$. This concept was later
explored by R.C.~James, and appeared under the name James
orthogonality in the literature on stable laws.

If $(\xi_1,\xi_2)$ are two jointly $\sas$ random variables with
$\alpha\in(1,2]$, then $\xi_2$ is said to be \emph{James orthogonal}
to $\xi_1$ (notation $\xi_2\dashv\xi_1$) if $\pnn{c\xi_1+\xi_2}\geq
\pnn{\xi_2}$ for all $c\in\R$. The James orthogonality condition can
be written as
\begin{displaymath}
  \pnn{u_1\xi_1+u_2\xi_2}\geq |u_2|\cdot\pnn{\xi_2}\,,
  \quad u=(u_1,u_2)\in\R^2\,.
\end{displaymath}
If $\alpha\in(1,2]$, we have $\xi_2\dashv\xi_1$ if and only if
$[\xi_1,\xi_2]_\alpha=0$, see \cite[Prop.~2.9.2]{sam:taq94}.

\begin{theorem}
  \label{thr:james-bivariate}
  If $(\xi_1,\xi_2)$ is $\sas$ in $\R^2$ with $\alpha\in(1,2]$ and the
  associated star body $F$, then $\xi_2\dashv\xi_1$ if and only if
  $F\subset \R\times[-a,a]$, where $a=\rho_F((0,1))$.
\end{theorem}
\begin{proof}
  Since the scale parameter of $(u_1\xi_1+u_2\xi_2)$ equals $h(K,u)$,
  the James orthogonality condition reads $h(K,u)\geq h(K,(0,u_2))$
  for all $u=(u_1,u_2)\in\R$. By passing to the radial function of
  $F=K^*$, we see that
  \begin{displaymath}
    \rho_F(u/\|u\|)\leq \frac{\|u\|}{|u_2|}\rho_F((0,1))\,.
  \end{displaymath}
  If $r(\theta)=\rho_F(\cos\theta,\sin\theta)$, then 
  \begin{displaymath}
    r(\theta)\leq |\sin\theta|^{-1}\rho_F((0,1))\,,
  \end{displaymath}
  which immediately implies the statement, taking into account the
  equation of $\R\times[-a,a]$ in polar coordinates. 
\end{proof}

Theorem~\ref{thr:james-bivariate} immediately implies that independent
$\sas$ variables are James orthogonal and that the James orthogonality
implies independence in the sub-Gaussian case, where $F$ is an
ellipsoid.

The isometry between $(S_\xi,\pnn{\cdot})$ and $(\R^d,\|\cdot\|_F)$
makes it possible to extend the James orthogonality concept for
$\alpha\in[1,2]$ (i.e. allow for $\alpha=1$) and immediately yields
the following result.

\begin{theorem}
  \label{thr:james-f}
  Let $\xi$ be $\sas$ with $\alpha\in[1,2]$ and associated star body
  $F$. For each $u,v\in\R^d$, we have $\langle
  \xi,u\rangle\dashv\langle\xi,v\rangle$ if and only if $u\dashv v$ in
  $(\R^d,\|\cdot\|_F)$.
\end{theorem}

Therefore, orthogonality property of $\sas$ random variables from
$\sS_\xi$ reduces to orthogonality in the normed space
$(\R^d,\|\cdot\|_F)$ if $F$ is convex. It is also possible to extend
the orthogonality concept for all $\alpha\in(0,2]$ as long as $F$ is
convex.

It is known that the orthogonality is symmetric in a normed space of
dimension at least 3 if and only if the space is Euclidean, i.e.  $F$
is an ellipsoid, see \cite[Th.~3.4.10]{thom96}.  The corresponding
probabilistic result is a part of \cite[Prop.~2.9.3]{sam:taq94}. In
dimension $d=2$ the orthogonality is symmetric if and only if the
boundary of $F$ is a \emph{Radon curve}, see \cite{day47} and
\cite[p.~94]{thom96}. The corresponding question for $\sas$ laws was
posed as an open problem in \cite[p.~109]{sam:taq94}. Recall that
$\partial F$ is a Radon curve if and only if the boundary of $F$ in
the second and fourth quadrants coincides with the boundary of the
projection body of $K=F^*$.

The James orthogonality is a property of the associated star body or
associated zonoid of an $\sas$ law and is not directly influenced by
$\alpha$. If it holds for an $\sas$ law, then it applies for all
symmetric stable laws that share the same associated star body.

It is also possible to define multivariate extensions of the James
orthogonality concept.

\begin{definition}
  \label{def:j-orth}
  If $\xi$ and $\eta$ are $\sas$ in $\R^d$ with $\alpha\in[1,2]$, then
  $\eta$ is said to be
  \begin{description}
  \item[(i)] James orthogonal to $\xi$ (notation $\eta\dashv\xi$) if the
    associated zonoid of $c\xi+\eta$ contains the associated zonoid of
    $\eta$ for all $c\in\R$;
  \item[(ii)] strongly James orthogonal to $\xi$ (notation
    $\eta\ddashv\eta$) if $\langle v,\eta\rangle$ is James orthogonal
    to $\langle u,\xi\rangle$ for all $u,v\in\R^d$.
  \end{description}
\end{definition}

The strong James orthogonality is linear invariant, i.e.  all linear
transformations preserve this property.  It is easy to see that if
$\eta\ddashv\xi$, then the associated zonoid of $(c\xi+\eta)$ contains
the associated zonoid of $\eta$ for all $c\in\R$, i.e. the strong
orthogonality implies (i). For this it suffices to note that this
associated zonoid has the support function $h(K,(cx,x))$ and apply
Definition~\ref{def:j-orth}(ii) with $u=cx$ and $v=x$.

\begin{theorem}
  \label{thr:james-multiv}
  If $(\xi,\eta)$ is $\sas$ in $\R^{2d}$ with $\alpha\in[1,2]$ and the
  associated star body $F$, then $\eta\ddashv\xi$ if and only if
  \begin{equation}
    \label{eq:u+v_fgeq-v_f-}
    \|u+v\|_F\geq \|v\|_F
  \end{equation}
  for all $u=(u_1,\dots,u_d,0,\dots,0)$ and
  $v=(0,\dots,0,v_1,\dots,v_d)$. Furthermore, $\eta\dashv\xi$ if and
  only if (\ref{eq:u+v_fgeq-v_f-}) holds for
  $u=(c,\dots,c,0,\dots,0)$, $v=(0,\dots,0,1,\dots,1)$ and all
  $c\in\R$. 
\end{theorem}
\begin{proof}
  For each $u',u''\in\R^d$, the scale parameter of $c\langle
  u',\xi\rangle+\langle u'',\eta\rangle$ is $\|(cu',u'')\|_F$. By the
  condition, this is at least $\|(0,u'')\|_F$, which is the scale
  parameter of $\langle u'',\eta\rangle$. 
\end{proof}

If $\xi$ and $\eta$ from Theorem~\ref{thr:james-multiv} are
independent, then $F$ is the $\alpha$-star sum of the associated star
bodies of $\xi$ and $\eta$, i.e. 
\begin{displaymath}
  \|u+v\|_F^\alpha=\|u+v\|_{F_\xi}^\alpha+\|u+v\|_{F_\eta}^\alpha
  =\|u\|_{F_\xi}^\alpha+\|v\|_{F_\eta}^\alpha
  \geq \|v\|_{F_\eta}^\alpha\,,
\end{displaymath}
i.e. $\eta$ is strong James orthogonal to $\xi$.

\section*{Acknowledgements}
\label{sec:acknowledgements}

This work was supported by the Swiss National Science Foundation Grant
No.~200020-111779.

\newcommand{\noopsort}[1]{} \newcommand{\printfirst}[2]{#1}
  \newcommand{\singleletter}[1]{#1} \newcommand{\switchargs}[2]{#2#1}


\end{document}